\documentclass[a4paper,12pt]{amsart}
\usepackage{amsmath, amssymb, amsthm, xcolor,enumerate}
\usepackage{verbatim}
\usepackage[backref=page]{hyperref}
\usepackage{xypic}
\usepackage{graphicx}
\usepackage{mathtools}

\usepackage{color}
\definecolor{cutcolour}{RGB}{0,100,0}

\input{xy} 
\xyoption{all} 
\CompileMatrices

\newcounter{theorem}
\newtheorem{theorem}[theorem]{Theorem}
\newtheorem*{theorem*}{Theorem}
\newtheorem{lemma}[theorem]{Lemma}
\newtheorem{proposition}[theorem]{Proposition}
\newtheorem{corollary}[theorem]{Corollary}

\theoremstyle{definition}
\newtheorem{definition}[theorem]{Definition}

\newtheorem{conjecture}[theorem]{Conjecture}
\newtheorem{question}[theorem]{Question}
\newtheorem*{question*}{Question}

\newtheorem*{remark*}{Remark}
\newtheorem{remark}[theorem]{Remark}
\newtheorem{example}[theorem]{Example}

\numberwithin{equation}{section}

\newcommand{\Q}{\mathcal Q}
\newcommand{\Z}{\mathcal Z}
\newcommand{\N}{\mathbb N}

\newcommand{\tr}{\mathrm{tr}}
\newcommand{\Cu}{\mathrm{Cu}}

\newcommand{\M}{\mathcal{M}}
\renewcommand{\S}{\mathcal{S}}

\title[Uniform property $\Gamma$]{Uniform property $\Gamma$}

\author[J.\ Castillejos]{Jorge Castillejos}
\address{\hskip-\parindent Jorge Castillejos, Department of Mathematics, KU Leuven, Celestijnenlaan 200b, 3001 Leuven, Belgium.}
\curraddr{Institute of Mathematics, Polish Academy of Sciences, ul. {\'S}niadeckich 8, 00-656 Warszawa, Poland}
\email{jcastillejoslopez@impan.pl}
\author[S.\ Evington]{Samuel Evington}
\address{\hskip-\parindent Samuel Evington, Mathematical Institute, University of Oxford, Oxford, OX2 6GG, UK.}
\email{Samuel.Evington@maths.ox.ac.uk}
\author[A.\ Tikuisis]{Aaron Tikuisis}
\address{\hskip-\parindent Aaron Tikuisis, Department of Mathematics and Statistics, University of Ottawa, Ottawa, K1N 6N5, Canada.}
\email{aaron.tikuisis@uottawa.ca}
\author[S.\ White]{Stuart White}
\address{\hskip-\parindent Stuart White, Mathematical Institute, University of Oxford, Oxford, OX2 6GG, UK.}
\email{stuart.white@maths.ox.ac.uk}
\thanks{Research partially supported by: an Alexander von Humboldt Foundation fellowship (SW); European Research Council Consolidator Grant 614195 RIGIDITY (JC);  EPSRC (EP/N00874X/1 (AT); EPSRC:EP/R025061/1 (SE, SW); long term structural funding -- a Methusalem grant of the Flemish Government (JC); NCN (2014/14/E/ST1/00525) (SE); NSERC (AT)}

\begin{document}
\maketitle

\begin{abstract}
We further examine the concept of uniform property $\Gamma$ for $C^*$-algebras introduced in our joint work with Winter. In addition to obtaining characterisations in the spirit of Dixmier's work on central sequences in II$_1$ factors, we establish the equivalence of uniform property $\Gamma$, a suitable uniform version of McDuff's property for $C^*$-algebras, and the existence of complemented partitions of unity for separable nuclear $C^*$-algebras with no finite dimensional representations and a compact (non-empty) tracial state space.  As a consequence, for $C^*$-algebras as in the Toms--Winter conjecture, the combination of strict comparison and uniform property $\Gamma$ is equivalent to Jiang--Su stability.  We also show how these ideas can be combined with those of Matui--Sato to streamline Winter's classification by embeddings technique.
\end{abstract}

\renewcommand*{\thetheorem}{\Alph{theorem}}

\section*{Introduction}

The study of approximately central sequences associated to operator algebras is almost as old as the subject itself.  In their foundational work (\cite{MvN43}), Murray and von Neumann distinguished factors associated to free groups from the hyperfinite II$_1$ factor; they did this through Property $\Gamma$ --- the existence of approximately central unitaries of trace zero --- showing the latter has this property, while the former do not.  Property $\Gamma$ was subsequently analysed in detail by Dixmier (\cite{Di69}), and shown to be equivalent to non-triviality of the central sequence algebra. Moreover, in this case the central sequence algebra is necessarily diffuse, so there are approximately central projections taking all possible values of the trace. This latter characterisation has been critical in a range of applications from cohomology to the generator problem (\cite{Chr86,GePopa,Pi01,CPSS03}). On the other hand, failure of property $\Gamma$ is equivalent to fullness, and normally used through a spectral gap characterisation of Connes \cite[Theorem 2.1]{Co76}, which provides the starting point of his celebrated proof that injective II$_1$ factors are hyperfinite by enabling property $\Gamma$ to be obtained from semidiscreteness. This is the fount of central sequences, and the provision of these to a sufficient degree to show that an injective II$_1$ factor is McDuff (i.e., absorbs the hyperfinite II$_1$ factor tensorially) is the goal of the next stages of Connes' argument. Outside the injective setting, both the presence and absence of (relative versions of) property $\Gamma$ remain key topics. Spectral gap methods (often applied to subalgebras, where these can be viewed as controlling the location of approximately commuting sequences for an inclusion in the spirit of a relative notion of property $\Gamma$) have subsequently been a key tool in Popa's deformation-rigidity theory (see \cite{Po07}, for example), and progress has recently been made on the possible structure of central sequences in general II$_1$ factors  (\cite{IS}).

This paper focuses on the uniform version of property $\Gamma$ for $C^*$-algebras introduced in our recent work with Winter (\cite{CETWW}). \emph{Uniform property $\Gamma$} requires that we can approximately divide elements in a $C^*$-algebra in a central fashion and uniformly in all traces; this amounts to simultaneously (and uniformly) witnessing that every II$_1$ factor representation has Dixmier's formulation of property $\Gamma$. Uniform property $\Gamma$ differs from both the pointwise version of property $\Gamma$ for $C^*$-algebras used to obtain similarity degree estimates\footnote{This asks for every II$_1$ factor representation to have property $\Gamma$, but does not require any compatibility between how property $\Gamma$ occurs in these representations.} (\cite{QS16}), and from a version of property $\Gamma$ used to preclude certain tensorial absorption phenomena (\cite{GJS00}).\footnote{This is a modification of Murray and von Neumann's original definition, asking for approximately central unitaries which are uniformly zero in all traces; the main application is to show that $C^*$-algebras with unique trace whose II$_1$ factor representation does not have property $\Gamma$ cannot absorb the Jiang--Su algebra tensorially.}

Our motivation comes from the Toms--Winter regularity conjecture, which seeks to identify through abstract conditions those simple nuclear $C^*$-algebras which are accessible to classification by $K$-theory and traces (see \cite[Section 5]{Wi18}, for example, for a full discussion of this conjecture, which we also discuss further in Section \ref{sec:TW}).  A general local-to-global argument over the trace simplex (see \cite[Lemma 4.1]{CETWW} for a precise statement) provides the key ingredient in the passage from tensorial absorption to finite nuclear dimension in \cite{CETWW}, and it turns out that uniform property $\Gamma$ is the natural abstract condition which facilitates this. Our purpose here is to undertake an in-depth study of uniform property $\Gamma$ in its own right and set out its relation to the Toms--Winter conjecture for simple, separable, nuclear $C^*$-algebras.  

In Section \ref{sec:EquivalentFormulations}, we give a number of equivalent reformulations of uniform property $\Gamma$ in the spirit of Dixmier's original work for factors. In Section \ref{sec:Bauer}, we discuss the subtle issue of tracial factorisation: if $(e_n)_{n=1}^\infty$ is a central sequence in a unital  $C^*$-algebra $A$, and $a\in A$, what can we say about $\tau(ae_n)$ for a trace $\tau$ on $A$?  This leads to a simplification of the definition of uniform property $\Gamma$ in the case when the trace space of $A$ is a Bauer simplex. In this case, one only has to ask for the unit of $A$ to be approximately centrally divisible uniformly in trace (Corollary \ref{cor:ReductionToAis1}).\footnote{As just stated, $A$ needs to be unital, but the more precise statement in Corollary \ref{cor:ReductionToAis1} does not require this.}

The next phase of the paper is devoted to the role of uniform property $\Gamma$ in the structure theory of simple separable nuclear $C^*$-algebras. In Section \ref{sec:CPoU}, we provide a converse to the main technical result of \cite{CETWW}, and using Matui and Sato's breakthrough \cite{MS12}, characterise $\Z$-stability for simple, separable, nuclear $C^*$-algebras with non-empty and compact tracial state space, as the combination of strict comparison and uniform property $\Gamma$.  This can be thought of as analogous to Connes' passage from property $\Gamma$ to McDuff in his work on injective factors.  Indeed, we formalise the notion of a $C^*$-algebra being uniformly McDuff, a concept which implicitly played major roles in \cite{Wi12} and \cite{MS12}. Then Theorem \ref{GammaMcDuff} provides the passage from uniform property $\Gamma$ to a uniform McDuff property for non-elementary, separable, nuclear $C^*$-algebras.  In particular, this sheds light on the remaining open implication in the Toms--Winter conjecture: the missing piece is uniform property $\Gamma$. We discuss the role of property $\Gamma$, and divisibility conditions more generally, in the Toms-Winter conjecture in Section \ref{sec:TW}. Here, we highlight the following statement, extracted from Theorem \ref{thm:ZstableIffGammaComparison}, which combines the work in Section \ref{sec:CPoU}, with a wealth of developments over the last decade including \cite{Wi12,MS12,MS14,KR14,SWW15,BBSTWW,CETWW}.

\begin{theorem}\label{thm:TomsWinterAssumingGamma}
	The Toms--Winter conjecture holds among separable, simple, unital, nuclear, non-elementary $C^*$-algebras which have uniform property $\Gamma$.
\end{theorem}

Examples of separable nuclear $C^*$-algebras with property $\Gamma$ are by now abundant.  Kerr and Szab\'o establish uniform property $\Gamma$ for crossed product $C^*$-algebras arising from a free action with the small boundary property of an infinite amenable group on a compact metrisable space in \cite[Theorem 9.4]{KS18}.  Combining Theorem \ref{thm:TomsWinterAssumingGamma} with this result, the Toms--Winter conjecture then holds for the simple $C^*$-algebras one obtains from free minimal actions with the small boundary property --- minimality giving rise to simplicity of the crossed product. This was also recorded in \cite{KS18}, which used an early version of this paper for Theorem \ref{thm:TomsWinterAssumingGamma}.

\begin{corollary}[{c.f.\ \cite[Corollary 9.5]{KS18}}]
Let $G\curvearrowright X$ be a free minimal action of an infinite amenable group on a compact Hausdorff space with the small boundary property.  Then the Toms--Winter conjecture holds for $C(X)\rtimes G$.
\end{corollary}

 With hindsight, the papers \cite{Sa12,KR14,TWW15} all establish uniform property $\Gamma$ for separable nuclear $C^*$-algebras whose extremal traces are compact and finite dimensional. This is why strict comparison implies Jiang--Su-stability for these algebras.  But uniform property $\Gamma$ holds outside this situation too: various diagonal $AH$-algebras, including Villadsen algebras of infinite nuclear dimension have uniform property $\Gamma$ (see Proposition \ref{prop:DiagonalAHGamma}).  Indeed, it is open whether every simple, separable, unital, infinite dimensional nuclear $C^*$-algebra $A$ has uniform property $\Gamma$.

\begin{question}\label{QB}
Does every separable nuclear $C^*$-algebra with no (non-zero) finite dimensional representations and non-empty compact set of tracial states have uniform property $\Gamma$?
\end{question}

Using Theorem \ref{thm:TomsWinterAssumingGamma} (and in particular Matui and Sato's work \cite{MS12}), a positive answer to Question \ref{QB} would resolve the Toms--Winter conjecture affirmatively.  Moreover, in the case where $T(A)$ is a Bauer simplex, Question \ref{QB}, can be formulated using the language of $W^*$-bundles from \cite{Oz13}: is the $W^*$-bundle obtained as the strict closure of $A$ trivial? This is a bundle over the extreme boundary of $\partial_eT(A)$, with fibres the hyperfinite II$_1$ factor, and it remains an open question, first raised implicitly just prior to Corollary 16 in \cite{Oz13} and explicitly as \cite[Question 3.14]{BBSTWW}, whether every $W^*$-bundle over a compact metrisable space with hyperfinite II$_1$ factor fibres is necessarily trivial.  By Ozawa's trivialisation theorem (\cite[Theorem 15]{Oz13}), this happens when the $W^*$-bundle has an appropriate version of uniform property $\Gamma$ (which in the case of bundles coming from strict closures of $C^*$-algebras is precisely uniform property $\Gamma$ for the $C^*$-algebra).

We end the paper in Section \ref{sec:CBErevisited}, by turning to the Elliott classification programme. In particular we examine how the techniques of \cite{CETWW} powered by uniform property $\Gamma$ simplify a component of the classification theorem \cite[Corollary D]{TWW17}. Specifically, Theorem \ref{ClassEmbeddings} eliminates the hypothesis of finite nuclear dimension from Winter's classification-by-embeddings theorem in \cite{Wi16}, which plays a fundamental role in the tracial approximation approach to classification in \cite{GLN,EGLN15}.  The proof, which is heavily inspired by Matui and Sato's work \cite{MS14}, has parallels to an approach to the final steps in Connes' argument (see Remark \ref{rem:Connes}).

\subsection*{Acknowledgements} We thank Leonel Robert and Hannes Thiel for several helpful conversations regarding the Cuntz semigroup, and in particular for helping us with Proposition \ref{prop:Divisibility}.  We also thank Wilhelm Winter for the discussions in our collaboration \cite{CETWW} which sparked this paper, and Chris Schafhauser and G\'abor Szab\'o for helpful comments on a preprint version of this paper. Finally we thank the referees for their helpful comments on the first version of the paper.

\numberwithin{theorem}{section}
\section{Preliminaries}
\label{sec:Prelims}

Throughout the paper, a \emph{trace} on a $C^*$-algebra $A$ means a tracial state. We write $T(A)$ for the set of traces on $A$ endowed with the weak$^*$-topology.  Our framework is the setting where $T(A)$ is non-empty and compact (as is the case when $A$ is unital and has a trace), so that $T(A)$ becomes a Choquet simplex \cite[Theorem 3.1.14]{Sak98} which is metrisable when $A$ is separable. We will often at least implicitly assume that $A$ has no (non-zero) finite dimensional representations, as this will be an immediate consequence of the definition of uniform property $\Gamma$.  This framework roughly corresponds to the II$_1$ factor setting where property $\Gamma$ first originated. Moreover, unital, stably finite, simple, exact $C^*$-algebras always fall into the scope of the paper (using a deep result of Haagerup \cite{Ha14} to see that $T(A)\neq \emptyset$).  In the non-unital case, the tracial state space assumptions are a genuine restriction. But given a simple, separable $C^*$-algebra $A$ which has non-zero densely defined lower semicontinuous tracial weights, one can often extract a hereditary subalgebra $A_0$ of $A$ whose tracial state space is non-empty and compact (see for example \cite[Lemma 2.5]{CE}; the hypothesis of this lemma can be verified using $\mathcal Z$-stability or, using the very recent results of \cite[Theorem 7.5]{APRT}, through stable rank one).  In this case, for many purposes (including notions of regularity in the Toms-Winter conjecture, or tackling questions of classification), one can work with this instead.

Let $A$ be a $C^*$-algebra with $T(A) \neq \emptyset$. Each $\tau \in T(A)$ induces a seminorm given by $\|a\|_{2,\tau} \coloneqq  \tau(a^*a)^{1/2}$ for $a\in A$. Given a non-empty subset $X\subset T(A)$, we obtain a \emph{uniform $2$-seminorm} associated to $X$ by
\begin{equation}
	\|a\|_{2,X} \coloneqq  \sup_{\tau \in X} \|a\|_{2,\tau},\quad a\in A.
\end{equation}
The seminorm $\|\cdot\|_{2,X}$ is dominated by the operator norm $\|\cdot\|$, and we have 
\begin{equation}
	\|ab\|_{2,X} \leq \min \lbrace\|a\|\|b\|_{2,X},\|b\|\|a\|_{2,X}\rbrace,\quad a,b\in A. \label{2normMultiplicative}   
\end{equation} 
We will primarily be interested in the case that $X=T(A)$, when the seminorm $\|\cdot\|_{2,T(A)}$ is known as the \emph{uniform trace seminorm}.  In many situations of interest this seminorm is actually a norm.\footnote{For example, if $A$ is unital, simple, stably finite, and exact, then $\|\cdot\|_{2,T(A)}$ is a norm.} In general, the $\|\cdot\|_{2,T(A)}$-null elements of $A$ form an ideal by \eqref{2normMultiplicative} and $\|\cdot\|_{2,T(A)}$ descends to a norm on the quotient. 

We will work with two flavours of central sequence algebras in the paper.  To set this up, we let $\omega \in \beta\N\setminus \N$ denote a fixed free ultrafilter. The ultrapower of a $C^*$-algebra $A$ is defined by
\begin{equation}
	A_\omega \coloneqq  \frac{\ell^\infty(A)}{\lbrace (a_n)_{n=1}^\infty: \lim_{n\rightarrow\omega}\|a_n\| = 0\rbrace}.
\end{equation} 
When $T(A)\neq\emptyset$, the \emph{uniform tracial ultrapower} of $A$ is defined by
\begin{equation}
	A^\omega \coloneqq  \frac{\ell^\infty(A)}{\lbrace (a_n)_{n=1}^\infty: \lim_{n\rightarrow\omega}\|a_n\|_{2,T(A)} = 0\rbrace}.
\end{equation}
Since $\|a\|_{2,T(A)} \leq \|a\|$ for all $a \in A$, there are canonical surjections $\ell^\infty(A) \twoheadrightarrow A_\omega \twoheadrightarrow A^\omega$.  As is standard, we use representative sequences in $\ell^\infty(A)$ to denote elements of $A_\omega$ and $A^\omega$, respectively.  Recall that a key property of ultrapowers is countable saturation, which loosely speaking enables one to pass from approximate satisfaction of certain properties to exact satisfaction.  The precise technical tool we use to do this is Kirchberg's $\epsilon$-test (\cite[Lemma A.1]{Kir06}). 

Given a sequence $(\tau_n)_{n=1}^\infty$ in $T(A)$, one can form a trace on $\ell^\infty(A)$ by
\begin{equation}
	(a_n)_{n=1}^\infty \mapsto \lim_{n\rightarrow\omega} \tau_n(a_n),
\end{equation}
which induces traces on the ultrapowers $A_\omega$ and $A^\omega$. Traces on $A_\omega$ and $A^\omega$ arising in this manner are called \emph{limit traces}. We write $T_\omega(A)$ for the set of limit traces on either $A_\omega$ or $A^\omega$.  Notice that essentially by construction $\|\cdot\|_{2,T_\omega(A)}$ is a norm on $A^\omega$.

We identify $A$ with the subalgebra of $A_\omega$ coming from constant sequences. When $T(A)$ is non-empty, we have a canonical map $\iota:A\rightarrow A^\omega$ given by composing this identification with the quotient $A_\omega \to A^\omega$; $\iota$ is an embedding when $\|\cdot\|_{2,T(A)}$ is a norm. 

By \cite[Proposition 1.11]{CETWW}, compactness of the trace space precisely characterises unitality of the uniform tracial ultrapower. This is the fundamental  reason behind our choice of framework for the study of uniform property $\Gamma$.  In order to collect a fact for later use, we give an alternative proof of the unitality of $A^\omega$ below.

\begin{proposition}\label{UnitalProp}
Let $A$ be a separable $C^*$-algebra with $T(A)$ compact and non-empty.  Then $A^\omega$ is unital. Moreover, if $(e_n)_{n=1}^\infty$ is an approximate unit for $A$, then $\|\iota(e_n)-1_{A^\omega}\|_{2,T_\omega(A)}\rightarrow 0$.
\end{proposition}
\begin{proof}
Let $(e_n)_{n=1}^\infty$ be an approximate unit for $A$. Then by Dini's Theorem, $\lim_{n\to\omega} \inf_{\tau \in T(A)} \tau(e_n) = 1$.

Then for any $(x_n)_{n=1}^\infty\in\ell^\infty(A)$, working in the minimal unitisation of $A$ we have,
\begin{equation}
\sup_{\tau\in T(A)}\|e_nx_n-x_n\|_{2,\tau}\leq \|x_n\|^{1/2}\sup_{\tau\in T(A)}(1-\tau(e_n))^{1/2}\rightarrow 0,
\end{equation}
as $n\rightarrow\omega$. Therefore $A^\omega$ is unital and the unit is represented by the sequence $(e_n)_{n=1}^\infty$. Now
\begin{align}
\|1_{A^\omega}-\iota(e_n)\|_{2,T_\omega(A)}^2&=\lim_{m\to\omega} \sup_{\tau \in T(A)}\tau((e_m-e_n)^2)\nonumber\\&\leq 1-\inf_{\tau \in T(A)}\tau(e_n)\rightarrow 0,
\end{align}
as $n\rightarrow\omega$, as claimed.
\end{proof}

The norm central sequence algebra is $A_\omega\cap A'$, and by abuse of notation we write $A^\omega\cap A'$ for the \emph{uniform tracial central sequence algebra}; strictly speaking, when $\|\cdot\|_{2,T(A)}$ is not a norm, $A^\omega\cap A'$ is shorthand for $A^\omega\cap \iota(A)'$.  While it is immediate that $A_\omega$ quotients onto $A^\omega$, the corresponding result for central sequence algebras is deeper, and was established by Kirchberg and R\o{}rdam in \cite{KR14}, building on an observation of Sato \cite{Sa11}.  The result below is a combination of \cite[Proposition 4.5(iii) and Proposition 4.6]{KR14} (working in the minimal unitisation if $A$ is not unital).

\begin{lemma}[Central Surjectivity]\label{CentSurject}
Let $A$ be a separable $C^*$-algebra with $T(A)$ compact and non-empty. Then the canonical map $A_\omega\cap A'\rightarrow A^\omega\cap A'$ is a surjection.
\end{lemma}

For a $C^*$-algebra $A$ and a tracial state $\tau \in T(A)$, we let $\pi_\tau:A \to \mathcal B(\mathcal H_\tau)$ denote the GNS representation.  Recall (for context only) that $\pi_\tau(A)''$ is a factor if and only if $\tau$ is an extremal trace (\cite[Th\'eor\`eme 6.7.3]{Di77}). However in this paper, we often have to work with non-extremal traces and hence general finite von Neumann algebras. As every finite type I von Neumann algebra has a non-zero finite dimensional representation, GNS-representations associated to traces on $C^*$-algebras without finite dimensional quotients give rise to type II$_1$ von Neumann algebras.

\begin{proposition}\label{prop:TypeII}
Let $A$ be a $C^*$-algebra with no non-zero finite dimensional quotients and let $\tau \in T(A)$. Then $\pi_\tau(A)''$ is a type II$_1$ von Neumann algebra.
\end{proposition}

It is a consequence of Connes' celebrated equivalence of injectivity and hyperfiniteness (\cite[Theorem 6]{Co76}), that injective type II$_1$ factors are McDuff, in that they tensorially absorb the hyperfinite II$_1$ factor, or equivalently by McDuff's theorem from \cite{McD70}, admit $\|\cdot\|_2$-approximately central unital matrix embeddings.\footnote{It is of course immediate that the hyperfinite II$_1$ factor $\mathcal R$ is McDuff. However, in Connes' original argument that a separably acting injective II$_1$ factor $\mathcal M$ is isomorphic to $\mathcal R$, a key intermediate step is to show that $\mathcal M$ is McDuff.  So the McDuff property of injective II$_1$ factors is both a consequence, and an ingredient, of injectivity implies hyperfiniteness.} In Section \ref{sec:CPoU}, we require the analogous version of this fact for general injective type II$_1$ von Neumann algebras. This is well-known to experts, but does not appear to be in the literature, so we briefly sketch the details, starting by recording two standard facts regarding type II$_1$ von Neumann algebras. Firstly, these admit unital embeddings of matrix algebras; see, for example, \cite[Proposition V.1.35]{Tak79}, which proves the following lemma in the case $n=2$ (and whose proof can be modified to handle the general case).
\begin{lemma}\label{lem:UnitalEmbedding}
	Let $\mathcal M \neq 0$ be a type II$_1$ von Neumann algebra. Then there exists a unital embedding $M_n
 \rightarrow \mathcal M$ for all $n \in \mathbb{N}$. 
\end{lemma}

\begin{lemma}\label{Lem1.3}
	Let $\mathcal M$ be a type II$_1$ von Neumann algebra, and let $F$ be a finite dimensional subalgebra of $\mathcal M$. Then $\mathcal M\cap F'$ is type II$_1$.
\end{lemma}
\begin{proof}
Without loss of generality, assume $1_{\mathcal M} \in F$, and decompose $F$ as $\bigoplus_{k=1}^K M_{n_k}$.
Then, writing $e^{(k)}_{ij}$ for the matrix units of the $k$-th summand of $F$, one has
\begin{equation}
\mathcal M\cap F'\cong \bigoplus_{k=1}^K e^{(k)}_{11}\mathcal Me^{(k)}_{11};
\end{equation}
via the map $x \mapsto (e^{(1)}_{11}xe^{(1)}_{11},\dots,e^{(K)}_{11}xe^{(K)}_{11})$ with inverse $(y_1,\cdots, y_K)\mapsto \sum_{k=1}^K\sum_{i=1}^{n_k}e^{(k)}_{i1}y_ke^{(k)}_{1i}$.
The result then follows from the fact that $\mathcal M$ is type II$_1$.
\end{proof}

When $\mathcal M$ is a finite von Neumann algebra with a fixed faithful normal trace $\tau$, we use the notation $\mathcal M^\omega$ to denote the tracial von Neumann algebra ultrapower,
\begin{equation}
\label{eq:vNUltrapower}
\mathcal M^\omega\coloneqq \ell^\infty(\M)/\{(x_n)_{n=1}^\infty\in\ell^\infty(\M):\lim_{n\rightarrow\omega}\|x_n\|_{2,\tau}=0\},
\end{equation}
which is naturally equipped with the trace given on representative sequences by $\tau_\omega((x_n)_{n=1}^\infty)=\lim_{n\rightarrow\omega}\tau(x_n)$.  This is the same notation as we use for the uniform tracial ultrapower of a $C^*$-algebra, motivated by the special case when $A$ has a unique trace $\tau$, where the uniform tracial ultrapower $A^\omega$ is (by Kaplansky's density theorem) nothing but the tracial von Neumann algebra ultrapower $(\pi_\tau(A)'')^\omega$. It will always be clear from context when we are applying a tracial von Neumann algebra ultrapower rather than a uniform tracial ultrapower.

With this setup we now record the McDuff property of hyperfinite type II$_1$ von Neumann algebras in the form of approximately central matrix embeddings.  When we come to use this result in Section \ref{sec:CPoU}, we will obtain hyperfiniteness through Connes' celebrated theorem.

\begin{proposition}\label{prop:UltrapowerMatrixEmbeddings}
Let $\M$ be a hyperfinite type II$_1$ von Neumann algebra with separable predual and fixed faithful normal trace $\tau$. Then for each $n\in\mathbb N$, there exists a unital embedding $M_n
 \rightarrow \M^\omega \cap \M'$.
\end{proposition}
\begin{proof}
Fix $n \in \mathbb{N}$. By hyperfiniteness, say $\M = \left(\bigcup_{k=1}^\infty F_k\right)''$ where $(F_k)_{k=1}^\infty$ is an increasing sequence of finite dimensional subalgebras. For each $k \in \mathbb{N}$, by Lemmas \ref{lem:UnitalEmbedding} and \ref{Lem1.3}, there exists a unital embedding $\phi_k:M_n
 \rightarrow \M \cap F_k'$. The image of the induced map $\Phi:M_n
 \rightarrow \M^\omega$ commutes with each $F_k \subseteq \M^\omega$, so commutes with $\M \subseteq \M^\omega$.  
\end{proof}

\medskip

We end the preliminaries by recording two facts related to orthogonality and order zero maps.  Recall that positive elements in a $C^*$-algebra are said to be \emph{orthogonal} if $ab = 0$. We record the following elementary fact for future use. 

\begin{proposition}\label{lem:ProjectionsVsContractions}
	Let $A$ be a unital $C^*$-algebra. Given positive elements $e_1,\ldots,e_n \in A$ with $\sum_{i=1}^n e_i = 1_A$, the following are equivalent:
\begin{enumerate}
	\item The elements $e_1,\ldots,e_n$ are pairwise orthogonal.
	\item The elements $e_1,\ldots,e_n$ are projections.  
\end{enumerate}
\end{proposition}

A completely positive map $\phi:A \rightarrow B$ between $C^*$-algebras is said to be \emph{order zero} if it preserves orthogonality, i.e., $\phi(a)\phi(b) = 0$ whenever $ab = 0$; see \cite{WZ09} for the structure theory of order zero maps. For us, a key tool is the following order zero lifting theorem, essentially due to Loring (\cite[Theorem 4.9]{Lo93}), although we use a reformulation from \cite{Wi09}.

\begin{proposition}[Order zero lifting, {\cite[Proposition 1.2.4]{Wi09}}]
\label{prop:OrderZeroLifting}
Let $A$ be a $C^*$-algebra and $I \subseteq A$ an ideal.
Then for any $n \in \mathbb N$, every completely positive and contractive (c.p.c.) order zero map $M_n
 \to A/I$ lifts to a c.p.c.\ order zero map $M_n
 \to A$.
\end{proposition}

\section{Equivalent formulations of uniform property $\Gamma$}
\label{sec:EquivalentFormulations}

This section is concerned with establishing some equivalent formulations of uniform property $\Gamma$ in the spirit of Dixmier's original analysis of central sequence algebras of II$_1$ factors in \cite{Di69}. We begin by recalling the definition of uniform property $\Gamma$ from \cite[Definition 2.1]{CETWW}.

\begin{definition}[{\cite[Definition 2.1]{CETWW}}]\label{defn:Gamma}
Let $A$ be a separable $C^*$-algebra with $T(A)$ nonempty and compact.
Then $A$ is said to have \emph{uniform property $\Gamma$} if for all $n\in\mathbb N$, there exist projections $p_1,\dots,p_n\in A^\omega \cap A'$ summing to $1_{A^\omega}$,\footnote{And therefore pairwise orthogonal, as noted in Proposition \ref{lem:ProjectionsVsContractions}.} such that
\begin{equation}\label{gamma.def}
\tau(ap_i)=\frac{1}{n}\tau(a),\quad a\in A,\ \tau\in T_\omega(A),\ i=1,\dots,n.
\end{equation}
\end{definition}

The definition above does not depend on the choice of $\omega$, and indeed standard methods show that it has an equivalent local formulation.

\begin{proposition}\label{P2.2}
Let $A$ be a separable $C^*$-algebra with $T(A)$ nonempty and compact.
Then the following are equivalent:
\begin{enumerate}
\item $A$ has uniform property $\Gamma$. 
\item For any finite subset $\mathcal F\subset A$, any $\epsilon>0$, and any $n \in \mathbb N$, there exist pairwise orthogonal positive contractions $e_1,\dots,e_n \in A$ such that for $i=1,\dots,n$ and $a\in\mathcal F$, we have \begin{equation}\label{e2.2}
\|[e_i,a]\|_{2,T(A)} <\epsilon \text{ and } \sup_{\tau \in T(A)} |\tau(ae_i)-\frac1n\tau(a)| < \epsilon.
\end{equation}
\item 
For any finite subset $\mathcal F\subset A$, any $\epsilon>0$, and any $n \in \mathbb N$, there exist pairwise orthogonal positive contractions $e_1,\dots,e_n \in A$ such that for $i=1,\dots,n$ and $a\in\mathcal F$, we have 
\begin{equation}\label{e2.3}
\|[e_i,a]\|<\epsilon \text{ and } \sup_{\tau \in T(A)} |\tau(ae_i)-\frac1n\tau(a)| < \epsilon.
\end{equation}
\end{enumerate}
\end{proposition}
\begin{proof}
As $\|\cdot\|_{2,T(A)}\leq\|\cdot\|$, we have (iii)$\Rightarrow$(ii). Moreover, (ii)$\Rightarrow$(i) is a routine application of Kirchberg's $\epsilon$-test (\cite[Lemma A.1]{Kir06}) which we omit. If (i) holds, fix $n\in\mathbb N$, and $p_1,\dots,p_n\in A^\omega\cap A'$ witnessing uniform property $\Gamma$. Then we can use central surjectivity (Lemma \ref{CentSurject}) and liftability of orthogonal positive contractions (see \cite[Theorem 4.6]{Lo93}) to find pairwise orthogonal positive contractions $q_1,\dots,q_n\in A_\omega\cap A'$ lifting $p_1,\dots,p_n$.  These can in turn be lifted to pairwise orthogonal representing sequences $(e_{1,l})_{l=1}^\infty,\dots,(e_{n,l})_{l=1}^\infty$ in $\ell^\infty(A)$. Then the set of $l$ such that $e_{1,l},\dots,e_{n,l}$ satisfies (\ref{e2.3}) lies in $\omega$, and hence is non-empty, establishing (iii).
\end{proof}

The next result is our analogue of \cite[Proposition 1.10]{Di69}, which says that if a central sequence algebra of a II$_1$ factor is non-trivial, then it is automatically diffuse, i.e., it contains no minimal projections.  In particular, the a priori weaker condition (ii) ensures that it is only necessary to consider $n=2$ in the definition of uniform property $\Gamma$.

\begin{proposition}
\label{prop:GammaEquivalence}
Let $A$ be a separable $C^*$-algebra with $T(A)$ nonempty and compact.
The following are equivalent.
\begin{enumerate}
\item $A$ has uniform property $\Gamma$.
\item For some $0 < \lambda < 1$, there exists a projection $p \in A^\omega \cap A'$ such that
	\begin{equation}\label{eqn:Gamma3}
		\tau(ap)= \lambda\tau(a),\quad a\in A,\ \tau\in T_\omega(A).
	\end{equation}
\item For any $\|\cdot\|_{2,T_\omega(A)}$-separable subalgebra $S \subset A^\omega$, there is a unital $^*$-homomorphism $\phi:L^\infty([0,1]) \rightarrow A^\omega \cap S'$ such that
	\begin{equation}\label{eqn:Gamma2}
		\tau(a\phi(f))= \tau(a)\hspace*{-1mm} \int_0^1 \hspace*{-1mm} f(t)dt,\quad a\in S,\ f \in L^\infty([0,1]),\ \tau\in T_\omega(A).
	\end{equation}
\end{enumerate}
\end{proposition}
\begin{proof}
(i) $\Rightarrow$ (ii): Taking $n \coloneqq  2$ in Definition \ref{defn:Gamma}, we get a projection $p \coloneqq  p_1 \in A^\omega \cap A'$ such that \eqref{eqn:Gamma3} holds with $\lambda \coloneqq  \frac{1}{2}$.

(iii) $\Rightarrow$ (i): Take $S \coloneqq  A$ and let $\phi:L^\infty([0,1]) \rightarrow A^\omega \cap A'$ be a unital $^*$-homomorphism satisfying \eqref{eqn:Gamma2}.
Fix $n \in \mathbb{N}$. Set $p_i \coloneqq  \phi(\chi_{i})$, where $\chi_i$ is the indicator function of the interval $[(i-1)/n,i/n]$. Then $p_1,\ldots,p_n \in A^\omega \cap A'$ are projections satisfying \eqref{gamma.def}. 

(ii) $\Rightarrow$ (iii): Fix $\lambda$ such that (ii) holds and define a trace $\tau_\lambda:\mathbb{C}^2 \to \mathbb{C}$ by $\tau_\lambda(a,b)\coloneqq \lambda a + (1-\lambda) b$. Our first step is to show that (ii)$\Rightarrow$(ii$'$) below.
\begin{itemize}
	\item[(ii$'$)] For every $\|\cdot\|_{2,T_\omega(A)}$-separable subalgebra $S$ of $A^\omega$, there exists a unital $^*$-homomorphism $\phi_S:\mathbb{C}^2 \rightarrow A^\omega \cap S'$ such that
	\begin{equation}\label{prop:GammaEquivalence.new1}
		\tau(a\phi_S(b))= \tau(a)\tau_\lambda(b),\quad a\in S,\ b \in \mathbb{C}^2,\ \tau\in T_\omega(A).
	\end{equation}
\end{itemize} 
Indeed, given such an $S$, let $(s^{(m)})_{m=1}^\infty$ be a $\|\cdot\|_{2,T_\omega(A)}$-dense sequence in $S$, and let $(s^{(m)}_n)_{n=1}^\infty$ be a representative sequence of $s^{(m)}$ for each $m$.
 Let $p$ be as in (ii) and let $(p_n)_{n=1}^\infty$ be a representative sequence of positive contractions for $p$ so that
\begin{align}
\lim_{n\rightarrow\omega}\|p_n-p_n^2\|_{2,T(A)}&=0, &&\notag \\
\notag
\lim_{n \rightarrow \omega} \|[p_n , a]  \|_{2,T(A)} &= 0, &&\text{and} \\
\lim_{n\rightarrow\omega}\sup_{\tau\in T(A)}|\tau(ap_n)-\lambda\tau(a)|&=0,&& a \in A.\qquad\qquad
\end{align}
Thus, for each $n\in\mathbb N$, we can find some $r(n)\in\mathbb N$, such that
\begin{align}
\notag
\|p_{r(n)}-p_{r(n)}^2\|_{2,T(A)}&\leq\frac{1}{n}, \\
\| [p_{r(n)}, s^{(m)}_n] \|_{2,T(A)} & \leq \frac{1}{n}, && \text{and}\notag \\
\sup_{\tau\in T(A)}|\tau(s^{(m)}_np_{r(n)})-\lambda\tau(s^{(m)}_n)|&\leq\frac{1}{n},&& m=1,\dots,n. 
\end{align}
Then the element $q$ in $A^\omega$ represented by $(p_{r(n)})_{n=1}^\infty$ is a projection in $A^\omega\cap S'$ with 
\begin{equation}
\tau(aq)=\lambda\tau(a),\quad a\in S,\ \tau\in T_\omega(A).
\end{equation}
Finally, we take $\phi_S:\mathbb C^2 \to A^\omega\cap S'$ to be the unital $^*$-homomorphism determined by $\phi_S(1,0) = q$. This establishes (ii$'$).

We now deduce (iii) from (ii$'$) by an iterative argument. To this end, fix a $\|\cdot\|_{2,T_\omega(A)}$-separable subalgebra $S \subset A^\omega$ as in part (iii), and define $S_0\coloneqq S$. For $i\geq 1$, recursively use condition (ii$'$) to construct $\phi_i$ satisfying (ii$'$) for the set $S_{i-1}$ and take $S_i\coloneqq C^*(S_{i-1}\cup \phi_i(\mathbb C^2))$. Since the $\phi_i$ commute by construction, together they define a unital $^*$-homomorphism $\phi:\bigodot_1^\infty \mathbb{C}^2 \rightarrow A^\omega \cap S'$ (where $\bigodot_1^\infty \mathbb C^2$ denotes the infinite algebraic tensor product of copies of $\mathbb C^2$) such that
\begin{equation} 
\label{eqn:GammaEquivalence1}
		\tau(a\phi(b))= \tau(a)\tau_\lambda^{(\infty)}(b),\quad a\in S,\ b \in \bigodot_1^\infty \mathbb{C}^2,\ \tau\in T_\omega(A),
\end{equation}
where $\tau_\lambda^{(\infty)}$ is the infinite product of $\tau_\lambda$.  Since the unit ball of $A^\omega$ is $\|\cdot\|_{2,T_\omega(A)}$-complete by \cite[Lemma 1.6]{CETWW}, so too is the unit ball of $A^\omega\cap S'$. Therefore we can extend $\phi$ by $\|\cdot\|_{2,\tau_\lambda}$-$\|\cdot\|_{2,T_\omega(A)}$-continuity (on bounded sets) to a unital $^*$-homomorphism from the tracial von Neumann algebra $(\mathcal M,\tau_\lambda^{(\infty)})\coloneqq \overline{\bigotimes}_1^\infty (\mathbb{C}^2,\tau_\lambda)$ with \eqref{eqn:GammaEquivalence1} extending to $b \in \mathcal M$.
As $0 < \lambda < 1$, $\M$ is diffuse. By uniqueness of the atomless standard probability space, $(\mathcal M,\tau_\lambda^\infty)$ is isomorphic to $L^\infty([0,1])$ (with Lebesgue measure). This gives \eqref{eqn:Gamma2}. 
\end{proof}

\begin{remark}
Note that the proof of Proposition \ref{P2.2} also gives local versions of condition (ii) above, e.g., a separable $C^*$-algebra $A$ with $T(A)$ non-empty and compact has uniform property $\Gamma$ if and only if for some $\lambda\in (0,1)$, for every finite subset $\mathcal F\subset A$ and every $\epsilon>0$, there is a positive contraction $e\in A$ such that\begin{align}
\|e^2-e\|_{2,T(A)}&<\epsilon,\ \|[e,a]\| <\epsilon,\ \text{and}\nonumber\\ \sup_{\tau \in T(A)} |\tau(ae)-\lambda\tau(a)| &< \epsilon,\ a\in\mathcal F.
\end{align}
\end{remark}

\section{Bauer simplices and tracial factorisation}
\label{sec:Bauer}

In this section, we set out how uniform property $\Gamma$ simplifies when $T(A)$ is a Bauer simplex, i.e., when the extreme boundary $\partial_eT(A)$ is compact (and non-empty).  Uniform property $\Gamma$ requires that arbitrary finite subsets of $A$ can be approximately divided in trace in an approximately central fashion.  When $T(A)$ is Bauer, it is only necessary to tracially divide the unit (as in Corollary \ref{cor:ReductionToAis1} below); since the division is tracial this can be phrased without explicit reference to a unit, and indeed it applies equally to non-unital $A$ (provided $T(A)$ is compact and non-empty).

 At the heart of this observation is the following `tracial factorisation' result for central sequences on extreme traces.  

\begin{proposition}[{cf.\ \cite[Lemma 4.2(i)]{Sa12} and \cite[Proposition 4.3.6]{Ev18}}]
\label{prop:BauerFactorization}
Let $A$ be a separable $C^*$-algebra with $T(A) \neq \emptyset$ and let $b=(b_n)_{n=1}^\infty \in A_\omega$.  If for some $\tau\in \partial_eT(A)$, we have $\lim_{n\rightarrow\omega} \|ab_n - b_na\|_{2,\tau} = 0$ for all $a \in A$, then 
\begin{equation}
\label{eq:BauerFactorization}
	\lim_{n \rightarrow \omega} |\tau(ab_n) - \tau(a)\tau(b_n)| = 0,
\end{equation}
for all $a \in A$.
Moreover, if $K \subseteq \partial_e T(A)$ is compact and we have $\lim_{n\to \omega} \|ab_n-b_na\|_{2,K} = 0$ for all $a \in A$, then
\begin{equation}
	\tau(ab)=\tau(a)\tau(b)
\end{equation}
holds for any limit trace $\tau \in T_\omega(A)$ arising from a sequence of traces in $K$, i.e., \eqref{eq:BauerFactorization} converges uniformly on $K$ (for fixed $a$).
\end{proposition}

\begin{proof}
We prove the second statement, as it evidently implies the first.
Without loss of generality, we assume that $b_n \in A_{+}$ for all $n \in \mathbb{N}$. Fix a compact subset $K \subseteq \partial_e T(A)$. Suppose there exists $a \in A$ for which it is not the case that
\begin{equation}
	\lim_{n \rightarrow \omega} \sup_{\tau \in K} |\tau(ab_n) - \tau(a)\tau(b_n)| = 0.
\end{equation} 
Then there exists $\epsilon > 0$ and a sequence $(\tau_n)_{n=1}^\infty$ in $K$ such that
\begin{equation}
\lim_{n \rightarrow \omega}|\tau_n(ab_n) - \tau_n(a)\tau_n(b_n)| \geq \epsilon.
\end{equation}
Let $\tau \in T_\omega(A)$ be the limit trace defined by the sequence $(\tau_n)_{n=1}^\infty$ and let $b \in A^\omega \cap A'$ be the element defined by the central sequence $(b_n)_{n=1}^\infty$, so that $\tau(ab) \neq \tau(a)\tau(b)$.

Let $\tau_A$ be the composition of $\tau$ with the diagonal map $A \to A^\omega$, i.e., $\tau_A = \lim_{n\to\omega} \tau_n$. Since $K$ is compact, $\tau_A \in K$. In particular, $\tau_A$ is an extremal tracial state. Next, consider the tracial functional $\sigma:A \rightarrow \mathbb{C}$ defined by $\sigma(x) = \tau(xb)$.\footnote{To check that $\sigma$ is tracial, use the fact that $b$ commutes with $A$.} We have $\sigma \leq \|b\|\tau_A$. Since $\tau_A$ is extremal, it follows that $\sigma = \alpha \tau_A$ for some positive constant $\alpha \geq 0$.\footnote{Set $\rho=\|b\|\tau_A-\sigma$, and then we may write $\|b\|\tau_A=\sigma+\rho$.
In the nontrivial case that both $\sigma,\rho$ are nonzero, $\tau_A$ can then be written as a convex combination 
\begin{equation}
 \tau_A = \frac{\|\sigma\|}{\|b\|}\cdot \frac{\sigma}{\|\sigma\|} + \frac{\|\rho\|}{\|b\|} \cdot \frac{\rho}{\|\rho\|}, 
\end{equation}
from which it follows that $\tau_A$ is a scalar multiple of $\sigma$.}

When $A$ is unital, it is immediate that $\alpha=\tau(b)$. In general, let $(e_m)_{m =1}^\infty$ be an approximate unit for $A$. By Dini's Theorem, $\rho(e_m) \nearrow 1$ uniformly for $\rho \in K$. We then compute that $|\tau_n(b_n)-\tau_n(e_mb_n)| \leq \|b\|(1-\tau_n(e_m))$. Thus, 
\begin{equation}
\tau(b) = \lim_{m\rightarrow\infty} \sigma(e_m) = \lim_{m\to\infty}\alpha\tau_A(e_m) = \alpha.
\end{equation}
Hence, $\tau(xb)=\tau(x)\tau(b)$ for all $x \in A$. Taking $x = a$ gives the required contradiction.
\end{proof}

Using the above proposition, we now show that in the Bauer case, it suffices to tracially divide the unit in an approximately central fashion to obtain uniform property $\Gamma$.

\begin{corollary}\label{cor:ReductionToAis1}
Let $A$ be a separable $C^*$-algebra with $T(A)$ a Bauer simplex. 
Then $A$ has uniform property $\Gamma$ if and only if for each $n\in \mathbb N$, there exist projections $p_1,\dots,p_n\in A^\omega \cap A'$ summing to $1_{A^\omega}$ such that
\begin{equation}\label{eqn:TraceHypothesis}
\tau(p_i)=\frac{1}{n},\quad \tau\in T_\omega(A),\ i=1,\dots,n.
\end{equation}
\end{corollary} 
\begin{proof}
	The forward implication is immediate.
For the converse, let $a \in A$ and $i \in \lbrace 1, \ldots, n \rbrace$, and we will show that $\tau(p_ia)=\frac1n\tau(a)$ for all $\tau \in T_\omega(A)$, from which it follows that $A$ has uniform property $\Gamma$.

Suppose $p_i$ is represented by $(b_k)_{k=1}^\infty \in \ell^\infty(A)$. Since $p_i \in A^\omega \cap A'$, we have $\lim_{k\rightarrow\omega} \|ab_k - b_ka\|_{2,T(A)} = 0$. Since $\partial_e T(A)$ is compact, we have 
\begin{equation}
 \lim_{k \rightarrow \omega} \sup_{\tau \in \partial_eT(A)} |\tau(ab_k) - \tau(a)\tau(b_k)| = 0
\end{equation} 
by Proposition \ref{prop:BauerFactorization}.  
		
From \eqref{eqn:TraceHypothesis}, we get that $\lim_{k\rightarrow\omega} \sup_{\tau \in T(A)}|\tau(b_k) - \frac{1}{n}| = 0$. Therefore, we have
\begin{equation}
	\lim_{k \rightarrow \omega} \sup_{\tau \in \partial_eT(A)} |\tau(ab_k) - \frac{1}{n}\tau(a)| = 0.
\end{equation}
Using convex combinations and continuity we see that for $\sigma \in T(A)$ and $k \in \mathbb N$, we have
\begin{equation}
 |\sigma(ab_k)-\frac1n\sigma(a)| \leq \sup_{\tau \in \partial_eT(A)} |\tau(ab_k)-\frac1n\tau(a)|, 
\end{equation}
and therefore,
\begin{equation}
	\lim_{k \rightarrow \omega} \sup_{\tau \in T(A)} |\tau(ab_k) - \frac{1}{n}\tau(a)| = 0.
\end{equation}  
Hence, $\tau(ap_i) = \frac{1}{n} \tau(a)$ for all $\tau \in T_\omega(A)$.
\end{proof}

Similarly, when $T(A)$ is a Bauer simplex, it suffices to consider the case $a\coloneqq 1_{A^\sim}$ in the equivalent formulations of uniform property $\Gamma$ established in Proposition \ref{prop:GammaEquivalence} (where $A^\sim$ denotes the minimal unitisation). 

The convergence in Proposition \ref{prop:BauerFactorization} need not be uniform over $\partial_e T(A)$ when $\partial_e T(A)$ is not compact, as the example below (which uses a very standard construction of a non-Bauer Choquet simplex, found for example, in \cite[Example 6.10]{Goodearl}) shows.

\begin{example}
\label{eg:NonCompactNonFactorization}
Consider the subhomogeneous $C^*$-algebra defined by
\begin{equation} A\coloneqq  \left\{(\lambda,\mu,(x_i)_{i=1}^\infty) \in \mathbb C \oplus \mathbb C \oplus \prod_{i=1}^\infty M_2
: \lim_{i\to\infty} x_i = \left(\begin{array}{cc} \lambda & 0 \\ 0 & \mu \end{array}\right)\right\}. \end{equation}
Every tracial state $\tau$ on $A$ is of the form 
\begin{equation}
	\tau(\lambda,\mu,(x_i)_{i=1}^\infty) = t_\lambda \lambda + t_\mu \mu + \sum_{i=1}^\infty t_i \tr_{M_2}(x_i) 
\end{equation}
for some $t_\lambda,t_\mu,t_i \in [0,1]$ satisfying $t_\lambda+t_\mu+\sum_{i=1}^\infty t_i = 1$. Therefore, the extreme traces on $A$ are given by
\begin{align}
	\tau_\lambda(\lambda,\mu,(x_i)_{i=1}^\infty) &\coloneqq  \lambda,\\
	\tau_\mu(\lambda,\mu,(x_i)_{i=1}^\infty) &\coloneqq  \mu, \nonumber\\
	\tau_n(\lambda,\mu,(x_i)_{i=1}^\infty) &\coloneqq  \tr_{M_2}(x_n), \quad n \in \mathbb{N}. \nonumber 
\end{align}
Note that $\partial_e T(A)$ is not compact, since
\begin{equation} \lim_{n\to\infty} \tau_n = \frac12(\tau_\lambda+\tau_\mu). \end{equation}

For each $n \in \N$, define $b_n \coloneqq  (\frac12,\frac12,(x_{n,i})_{i=1}^\infty)$ where
\begin{equation} x_{n,i} \coloneqq  \begin{cases} \frac12 1_{M_2}, \quad &i\neq n; \\ e_{11}, \quad &i= n,\end{cases} \end{equation}
and $e_{11}$ is the $(1,1)$-matrix unit.
Then for any $a = (\lambda,\mu,(y_i)_{i=1}^\infty) \in A$, we have
\begin{equation}
 \|[b_n,a]\| = \|[e_{11},y_n]\| \to \left\|\left[e_{11},\left(\begin{array}{cc} \lambda&0\\0&\mu \end{array}\right)\right]\right\| = 0, 
\end{equation}
as $n\to\infty$.
However, if we define
\begin{equation} a \coloneqq  (0,1,\left(\begin{array}{cc}0&0\\0&1\end{array}\right),\left(\begin{array}{cc}0&0\\0&1\end{array}\right),\dots)\in A, \end{equation}
then $\tau_n(ab_n)=\tr_{M_2}(0) = 0$ whereas $\tau_n(a)\tau_n(b_n) = \frac12\cdot\frac12 = \frac14$. Hence, the convergence in Proposition \ref{prop:BauerFactorization} is not uniform over $\partial_e T(A)$.
\end{example}

\begin{remark} 
Note in the example above, we even have $\tau(b_n)=\frac12$ for all $\tau \in T(A)$.
The $C^*$-algebra $A \otimes \mathcal Q$ (where $\mathcal Q$ is the universal UHF-algebra)  is an example of the same phenomenon that additionally has no finite-dimensional representations (and which certainly has property $\Gamma$, by virtue of being $\mathcal Q$-stable).\footnote{This is essentially immediate from $\mathcal Q$-stability; alternatively one can note that $\mathcal Q$-stable $C^*$-algebras are $\mathcal Z$-stable, where $\mathcal Z$ is the Jiang--Su algebra, and then appeal to \cite[Proposition 2.3]{CETWW}.} It is also natural to ask whether it is possible to observe this phenomenon with a simple $C^*$-algebra; in fact, any simple AF $C^*$-algebra with the same trace space (these exist by \cite{Bl80}) gives such an example. However, our approach to this relies on classification results for ``uniform trace norm completions'' of $\mathcal Z$-stable separable nuclear $C^*$-algebras which will be developed in forthcoming work of the authors together with Jos\'e Carri\'o{}n, James Gabe and Christopher Schafhauser, so we defer the details until then.
\end{remark}

However, Example \ref{eg:NonCompactNonFactorization} does not preclude the possibility that Corollary \ref{cor:ReductionToAis1} generalises to the case where $T(A)$ is a general Choquet simplex.

\begin{question}
Let $A$ be a separable $C^*$-algebra with no non-zero finite dimensional representations, and $T(A)$ compact and non-empty.  Suppose for some $n\in\mathbb N$, there are projections $p_1,\dots,p_n\in A^\omega\cap A'$ summing to $1_{A^\omega}$ with $\tau(p_i)=\tfrac1n$ for each $i$ and each $\tau\in T_\omega(A)$.  Must $A$ have uniform property $\Gamma$?
\end{question}

\section{Uniform property $\Gamma$ and complemented partitions of unity}
\label{sec:CPoU}

When we introduced uniform property $\Gamma$ with Winter in \cite{CETWW}, our motivation was to find an abstract general condition which gives rise to the key technical tool of \cite{CETWW} --- complemented partitions of unity (CPoU) --- which in turn is essential to passing from $\Z$-stability to finite nuclear dimension.  Our goal in this section is to prove the converse of \cite[Theorem 3.8]{CETWW} (and in fact obtain a stronger result), which will enable us to show how uniform property $\Gamma$ and other tracial divisibility conditions play a role in the Toms--Winter conjecture in the following section.  We recall the definition of CPoU below; see the discussion before and after \cite[Definition G]{CETWW} for the motivation behind this concept.

\begin{definition}[{\cite[Definition 3.1]{CETWW}}]
Let $A$ be a separable $C^*$-algebra with $T(A)$ non-empty and compact.  Then $A$ has CPoU if and only if, for every family of positive elements $a_1,\dots,a_k\in A^\omega$ and $\delta$ such that
\begin{equation}
\delta>\sup_{\tau\in T_\omega(A)}\min\{\tau(a_1),\dots,\tau(a_k)\},
\end{equation}
there exist pairwise orthogonal projections $p_1,\dots,p_k\in A^\omega \cap A'$ which sum to $1_{A^\omega}$ and have
\begin{equation}
\tau(p_ia_i)\leq \delta\tau(p_i),\quad i=1,\dots,n,\ \tau\in T_\omega(A).
\end{equation}
\end{definition}

The main technical argument in this section is to use CPoU to glue local occurrences of the McDuff property at each tracial fibre in Lemma \ref{lem:CPoUImpliesCentralDivis} to obtain the following global McDuff property.

\begin{definition}
Let $A$ be a separable $C^*$-algebra with $T(A)$ non-empty and compact. Say that $A$ is \emph{uniformly McDuff} if for each $n\in\mathbb N$, there exists a unital embedding $M_n\rightarrow A^\omega\cap A'$.
\end{definition}

\begin{remark}\label{GammaMcDuffRem}
\ 
\begin{enumerate}
    \item  Being uniformly McDuff is in fact a property of the strict closure of $A$,\footnote{We will develop this point further in a more abstract framework in forthcoming work with Jos\'e Carri\'o{}n, James Gabe and Christopher Schafhauser.}  in the spirit of McDuff's results \cite{McD70}.\footnote{Just as in the II$_1$ factor setting, it suffices to find a unital embedding $M_n\rightarrow A^\omega\cap A'$ for some $n\geq 2$ to obtain the uniform McDuff property.  Also, when $A$ is separable, the uniform tracial closure will absorb the hyperfinite II$_1$ factor tensorially (in a suitable tensor product) if and only if $A$ is uniformly McDuff.}
    \item The uniform McDuff property has repeatedly played a significant role in work on the Toms--Winter conjecture, albeit in a slightly different guise. By projectivity of order zero maps (Proposition \ref{prop:OrderZeroLifting}), and central surjectivity (Lemma \ref{CentSurject}), the uniform McDuff property is equivalent to the existence of order zero maps $\phi:M_n\rightarrow A_\omega\cap A'$ with $\tau(\phi(1))=1$ for all $\tau\in T_\omega(A)$.  Passing through local properties, this is equivalent to the same condition working with the classical sequence algebra $A_\infty\coloneqq \ell^\infty(A)/c_0(A)$, which was described in \cite[Definition 2.2]{TWW15} as ``$A$ admits uniformly tracially large c.p.c.\ order zero maps into $A_\infty\cap A'$''.  With the benefit of hindsight, we now prefer the terminology uniformly McDuff for these equivalent concepts. We discuss the role of the uniform McDuff property in the Toms--Winter conjecture in the next section.
\end{enumerate}
\end{remark} 

We now give our gluing lemma.  We could use CPoU via the local to global polynomial test of \cite[Lemma 4.1]{CETWW} to establish this, but prefer to show how CPoU is used directly.

\begin{lemma}
	\label{lem:CPoUImpliesCentralDivis}
	Let $A$ be a separable $C^*$-algebra with $T(A)$ non-empty and compact. Suppose $A$ has CPoU and that for every $\tau \in T(A)$, the GNS representation $\pi_\tau(A)''$ is McDuff, i.e., for each $n\in\mathbb N$, $M_n$ embeds unitally into $(\pi_\tau(A)'')^\omega \cap \pi_\tau(A)'$, where $(\pi_\tau(A)'')^\omega$ is the von Neumann algebra ultrapower (see \eqref{eq:vNUltrapower}) with respect to the canonical induced trace $\tau$ on $\pi_\tau(A)''$. Then $A$ is uniformly McDuff.
\end{lemma}
\begin{proof}
	Fix $n\in\mathbb N$, and fix a finite generating set $\mathcal F$ for $M_n$.
	To establish the result, it suffices to show that given a finite subset $\mathcal G$ of $A$ and $\epsilon>0$, there exists a c.p.c.\ order zero map $\phi:M_n \to A^\omega$ such that for all $x \in \mathcal F, b \in \mathcal G$, 
	\begin{equation} \|[\phi(x),b]\|_{2,T_\omega(A)} \leq \epsilon, \quad \text{and} \quad \|1_{A^\omega}-\phi(1_{M_n})\|_{2,T_\omega(A)}\leq 2\epsilon. \end{equation}
Once this is done, Kirchberg's $\epsilon$-test (\cite[Lemma A.1]{Kir06}) gives a unital c.p.c.\ order zero map $\psi:M_n\to A^\omega\cap A'$, and hence a $^*$-homomorphism.\footnote{That unital c.p.c.\ order zero maps are $^*$-homomorphisms is well known, and for example a direct consequence of the structure theorem for order zero maps (\cite[Theorem 2.3]{WZ09}), as recorded in \cite[Proposition 1.4]{BBSTWW}.}
	
	Therefore we fix a finite subset $\mathcal G$ of $A$ and a tolerance $\epsilon>0$.   Set 
	\begin{equation} \eta \coloneqq  \frac\epsilon{\sqrt{2(1+|\mathcal F|\cdot|\mathcal G|)}}. \end{equation}
	By taking a suitable element of an approximate unit for $A$, we can find a contraction $e\in A_+$ with $\|1_{A^\omega}-e\|_{2,T_\omega(A)}<\eta/2$ by Proposition \ref{UnitalProp}.
	
For the moment, fix $\tau \in T(A)$ and set $\M_\tau\coloneqq \pi_\tau(A)''$.   By hypothesis, there exists a unital embedding $\theta:M_n\to \M_\tau^\omega \cap \M_\tau'$. The $^*$-homomorphism $\ell^\infty(A) \rightarrow \M_\tau^\omega$ given by $(a_i)_{i=1}^\infty \mapsto (\pi_\tau(a_i))_{i=1}^\infty$ is surjective by the Kaplansky Density Theorem together with the existence of norm-preserving lifts. Therefore, by order zero lifting (Proposition \ref{prop:OrderZeroLifting}), there exists a sequence $(\theta_l:M_n
 \rightarrow A)_{l=1}^\infty$ of c.p.c.\ order zero maps which induces $\theta$. Setting $\phi_\tau \coloneqq  \theta_l$ for a suitable value of $l$ we can arrange to have
\begin{equation} \|[\phi_\tau(x),b]\|_{2,\tau} <\eta, \text{ and } 1-2\tau(\phi_\tau(1_{M_n}))+\tau(\phi_\tau(1_{M_n})^2)<\eta^2/4,
\end{equation}
for all $x \in \mathcal F, b \in \mathcal G$. It follows that for all $\sigma$ in a neighbourhood of $\tau$ in $T(A)$, we continue to have
	\begin{equation} \|[\phi_\tau(x),b]\|_{2,\sigma} <\eta, \text{ and } 1-2\sigma(\phi_\tau(1_{M_n}) )+\sigma(\phi_\tau(1_{M_n})^2)<\eta^2/4 \end{equation}
for all $x \in \mathcal F, b \in \mathcal G$.
	
We now use the compactness of $T(A)$ to find c.p.c.\ order zero maps $\phi_1,\dots,\phi_k:M_n \to A$ such that for every trace $\tau \in T(A)$ there exists $i\in\{1,\dots,k\}$ such that for all $x \in \mathcal F, b \in \mathcal G$,
	\begin{equation}
	\label{eq:CentralDivis1} \|[\phi_i(x),b]\|_{2,\tau} < \eta,  \text{ and } 1-2\tau(\phi_i(1_{M_n}))+\tau(\phi_i(1_{M_n})^2) < \eta^2/4.
	\end{equation}
Notice that, working in the minimal unitisation $A^\sim$, (\ref{eq:CentralDivis1}) implies that 
\begin{align}
\nonumber\|e-\phi_i(1_{M_n})\|_{2,\tau}&\leq \|e-1_{A^\sim}\|_{2,\tau}+\|1_{A^\sim}-\phi_i(1_{M_n})\|_{2,\tau}\nonumber\\
&<\eta/2+(1-2\tau(\phi_i(1_{M_n})+\tau(\phi_i(1_{M_n})^2))^{1/2}<\eta.
\end{align}

	For $i=1,\dots,k$, we now define
	\begin{equation} \label{l:CPoU-Gamma.ai}
	a_i \coloneqq  |e-\phi_i(1_{M_n})|^2 + \sum_{x \in \mathcal F} \sum_{b \in \mathcal G} |[\phi_i(x),b]|^2 \in A_+,
	\end{equation}
	and observe that when \eqref{eq:CentralDivis1} holds, we get
	\begin{align}
	\notag
	\tau(a_i) &= \|e-\phi_i(1_{M_n
})\|_{2,\tau}^2 + \sum_{x \in \mathcal F} \sum_{b \in \mathcal G} \|[\phi_i(x),b]\|_{2,\tau}^2 \\
	&< (1+|\mathcal F|\cdot|\mathcal G|)\eta^2= \epsilon^2/2.
	\end{align}
	Thus, we conclude that
	\begin{equation} \sup_{\tau\in T(A)}\min_{i=1,\dots,k}\tau(a_i) \leq \epsilon^2/2 < \epsilon^2. \end{equation}
	Viewing $a_1,\dots,a_k$ as elements of $A^\omega$ (i.e., abusing notation and writing $a_i$ for $\iota(a_i)$), we have 
	\begin{equation} 
	\sup_{\tau \in T_\omega(A)} \min_{i=1,\dots,k} \tau(a_i) < \epsilon^2.
	\end{equation}
	Using CPoU, 
	we obtain orthogonal projections $p_1,\dots,p_k \in A^\omega \cap A'$ summing to $1_{A^\omega}$ such that
	\begin{equation} 
	\label{l:CPoU-Gamma.3}
	\tau(a_ip_i) \leq \epsilon^2\tau(p_i),\qquad \tau \in T_\omega(A). \end{equation}
	Now define $\phi:M_n \to A^\omega$ by
	\begin{equation} \phi(x) \coloneqq  \sum_{i=1}^k p_i\phi_i(x)p_i,\quad x\in M_n. \end{equation}
	Since the $p_i$ are orthogonal and commute with the images of the $\phi_i$, this is an orthogonal sum of c.p.c.\ order zero maps, so it is itself c.p.c.\ order zero.
	For $x \in \mathcal F$, $b \in \mathcal G$, first compute
	\begin{align}
	[\phi(x),b] &= \sum_{i=1}^k [p_i\phi_i(x)p_i,b]= \sum_{i=1}^k p_i[\phi_i(x),b]p_i, \label{l:CPoU-Gamma.1}
	\end{align}
	using that $p_i$ commutes with $b$.
	Next, since the $p_i$ are orthogonal, it follows that
	\begin{eqnarray}
	|[\phi(x),b]|^2 
	&\stackrel{\eqref{l:CPoU-Gamma.1}}{=} &\sum_{i=1}^k  p_i[\phi_i(x),b]^*p_i[\phi_i(x),b]p_i \nonumber \\
	&=& \sum_{i=1}^k p_i|[\phi_i(x),b]|^2p_i \nonumber \\
	&\stackrel{\eqref{l:CPoU-Gamma.ai}}{\leq}& \sum_{i=1}^k p_ia_ip_i. \label{l:CPoU-Gamma.2}
	\end{eqnarray}
	Using this, for $\tau\in T_\omega(A)$, we get,
	\begin{eqnarray}
	\|[\phi(x),b]\|^2_{2,\tau} & = & \tau(|[\phi(x),b]|^2) \nonumber \\
	&\stackrel{\eqref{l:CPoU-Gamma.2}}{\leq}& \sum_{i=1}^k \tau(p_ia_ip_i) \nonumber \\
	&=& \sum_{i=1}^k \tau(a_ip_i) \nonumber \\
	&\stackrel{\eqref{l:CPoU-Gamma.3}}\leq& \sum_{i=1}^k \epsilon^2\tau(p_i) = \epsilon^2,
	\end{eqnarray}
	and thus $\|[\phi(x),b]\|_{2,T_\omega(A)}\leq \epsilon$.
	Similarly, for $\tau \in T_\omega(A)$, since $\sum_{i=1}^k p_i=1_{A^\omega}$, we have
	\begin{eqnarray}
	\|\phi(1_{M_n})-e\|^2_{2,\tau}
	&=& \tau\Big(\big(\sum_{i=1}^k p_i\phi_i(1_{M_n})-e\big)^2\Big) \nonumber \\
	&=& \tau\Big(\sum_{i=1}^k p_i(\phi_i(1_{M_n})-e)^2\Big) \nonumber \\
	&\stackrel{\eqref{l:CPoU-Gamma.ai}}{\leq}& \sum_{i=1}^k \tau(p_ia_i) \nonumber \\
	&\leq& \epsilon^2. 
	\end{eqnarray}
Since $\|1_{A^\omega}-e\|_{2,T_\omega(A)}<\eta<\epsilon$, we have $\|\phi(1_{M_n})-1_{A^\omega}\|_{2,T_\omega(A)}<2\epsilon$, as required.
\end{proof}

In particular, nuclear $C^*$-algebras with no finite dimensional quotients satisfy the fibrewise McDuff hypothesis of the previous lemma.

\begin{lemma}
	\label{cor:CPoUImpliesCentralDivis}
	Let $A$ be a separable, nuclear $C^*$-algebra with no finite dimensional quotients and with $T(A)$ nonempty and compact. Suppose $A$ has CPoU.
	Then $M_n$ 
embeds unitally into $A^\omega \cap A'$ for all $n \in \mathbb{N}$.
\end{lemma}

\begin{proof}
	For $\tau \in T(A)$, $\pi_\tau(A)''$ is a II$_1$ von Neumann algebra by Proposition \ref{prop:TypeII}.
	Nuclearity of $A$ implies hyperfiniteness of $\pi_\tau(A)''$ by Connes' theorem (\cite{Co76}), so by Proposition \ref{prop:UltrapowerMatrixEmbeddings}, $\pi_\tau(A)''$ is McDuff.
	Hence by Lemma \ref{lem:CPoUImpliesCentralDivis}, there exists a unital embedding $\phi:M_n
 \rightarrow A^\omega \cap A'$.
\end{proof}

The implication (iii)$\Rightarrow$(i) of the following theorem gives the converse to \cite[Theorem 3.8]{CETWW}.

\begin{theorem}\label{GammaMcDuff}
Let $A$ be a separable, nuclear $C^*$-algebra with no finite dimensional quotients and with $T(A)$ non-empty and compact.  Then the following are equivalent:
\begin{enumerate}
\item $A$ has CPoU.
\item $A$ is uniformly McDuff.
\item $A$ has uniform property $\Gamma$.
\end{enumerate}
\end{theorem}
\begin{proof}
(i)$\Rightarrow $(ii) was the previous lemma, and (iii)$\Rightarrow$(i) is \cite[Theorem 3.8]{CETWW}. The remaining implication, (ii)$\Rightarrow$(iii), is essentially immediate: as in the proof of \cite[Proposition 2.3]{CETWW}, given a unital embedding $\phi:M_n\rightarrow A^\omega\cap A'$, the elements $\phi(e_{11}),\dots,\phi(e_{nn})$ are pairwise orthogonal, sum to $1_{A^\omega}$ and, by uniqueness of the trace on $M_n$, satisfy 
	\begin{equation}
		\tau(\phi(e_{ii})a) = \frac{1}{n}\tau(a), \quad a \in A,\ i=1,\dots,n.\qedhere
	\end{equation}
\end{proof}

\section{The Toms--Winter Conjecture}
\label{sec:TW}

In this section, we discuss the remaining open implication in the Toms--Winter conjecture, and relate known results in this direction to uniform property $\Gamma$.

\begin{conjecture}[Toms--Winter]\label{TW}
Let $A$ be a separable, simple, nuclear, non-elementary $C^*$-algebra. Then the following are equivalent:
\begin{enumerate}[(i)]
\item $A$ has finite nuclear dimension.
\item $A\cong A\otimes\Z$, where $\Z$ is the Jiang--Su algebra of \cite{JS99}.
\item $A$ has strict comparison (described below).
\end{enumerate}
\end{conjecture}

By now conditions (i) and (ii) are known to be equivalent (by \cite{CETWW,CE,Wi12,Ti14}, building on \cite{MS14}). Moreover, in the presence of the Universal Coefficient Theorem, these two conditions characterise those separable simple nuclear $C^*$-algebras accessible to classification (\cite{Kir:Unpublished,Phi00,TWW15,EGLN15,GLN,CETWW}).  

The third condition, strict comparison, is a condition on positive elements in the stabilisation of $A$, which can be described in terms of the \emph{Cuntz semigroup}. Recall that the Cuntz semigroup, $\Cu(A)$, of a $C^*$-algebra $A$ is built from equivalence classes of positive elements of the stabilisation $A\otimes\mathcal K$ of $A$ as follows: for $a,b\in (A\otimes\mathcal K)_+$, write $a\precsim b$ if and only if there is a sequence $(x_n)_{n=1}^\infty$ in $A\otimes\mathcal K$ with $x_nbx_n^*\rightarrow a$, and define an equivalence relation $a\sim b$ if and only if $a\precsim b$ and $b\precsim a$.  Then $\Cu(A)\coloneqq (A\otimes\mathcal K)_+/\sim$.  This is an ordered abelian semigroup.\footnote{Addition of $x,y\in \Cu(A)$ is defined by $x+y=[a+b]$, where $a,b\in(A\otimes \mathcal K)_+$ are orthogonal representatives of $x$ and $y$ respectively.}  See \cite{APT11} for a survey.
A \emph{functional} on $\Cu(A)$ is an ordered semigroup homomorphism $\phi:\Cu(A)\rightarrow [0,\infty]$ which preserves increasing sequential suprema (when $A$ is unital, a functional is called \emph{normalised} if it maps the class of the unit to $1$), and the collection of functionals is denoted $F(\Cu(A))$.\footnote{In the unital case, functionals on $\Cu(A)$ come from 2-quasitraces as defined in \cite[Definition II.1.1]{BH82}, and by \cite{Ha14} these are the same as traces for exact $C^*$-algebras. Any $\tau\in T(A)$ (or more generally, any 2-quasitrace) extends uniquely to a densely defined lower semicontinuous trace (or 2-quasitrace) on $A\otimes \mathcal K$ (also denoted by $\tau$), and for $a\in (A\otimes\mathcal K)_+$, $d_\tau(a)\coloneqq \lim_{n\rightarrow\infty}\tau(a^{1/n})$ gives a well defined state on $\Cu(A)$; moreover every state arises in this fashion; see \cite{ERS}.}
A simple $C^*$-algebra has \emph{strict comparison} when one can deduce $x\leq y$ in $\Cu(A)$ from knowing that $\phi(x)\leq\phi(y)$ for all $\phi\in F(\Cu(A))$, with strict inequality whenever $\phi(y) \in (0,\infty)$. 
It is essentially a result of R\o{}rdam from \cite{Ro04} that simple $\mathcal Z$-stable $C^*$-algebras have strict comparison, and hence the implication (ii)$\Rightarrow$(iii) of Conjecture \ref{TW} was known to hold well before the conjecture was formulated.\footnote{There is a small detail to watch out for when looking at earlier papers involving the Cuntz semigroup, such as \cite{Ro04} and \cite{Wi12}. These often use the `incomplete' $W(A)\coloneqq \bigcup_{n=1}^\infty M_n(A)_+/\sim$ in place of $\Cu(A)=W(A\otimes\mathcal K)$ which came to the fore in \cite{CEI}.  } Consequently, the remaining open part of the Toms--Winter conjecture is the implication (iii)$\Rightarrow$(ii).  

In addition to the precise formulation of Conjecture \ref{TW}, research has also focused on trying to obtain \emph{any} Cuntz semigroup condition which characterises $\mathcal Z$-stability for simple, separable, unital and nuclear $C^*$-algebras.  Indeed, it has been suggested by Winter, at least as far back as the CBMS lecture series in 2011, that another potential condition is $\Cu(A) \cong \Cu(A\otimes \mathcal Z)$, which is certainly the strongest reasonable such candidate (see for instance \cite[Paragraph 5.4]{Wi18}).  Experts have known for some time that $\Cu(A) \cong \Cu(A\otimes \mathcal Z)$ represents the combination of strict comparison and an appropriate divisibility condition on the Cuntz semigroup. It is too much to ask for exact divisibility of $\Cu(A)$, which would mean that given any $x\in \Cu(A)$, and $k\in\mathbb N$, one can divide $x$ by $k$, i.e. find $y\in\Cu(A)$ with $ky=x$. For example, the unit of the Jiang--Su algebra is not divisible in the Cuntz semigroup by any $n\geq 2$. Thus, instead one asks for `almost-divisibility', the ability to divide elements into between $k$ and $k+1$ pieces.  There are a number of slightly differently formulations of this idea in the literature,\footnote{The concept of almost-divisibility had been around for some time before the present terminology stuck, showing up implicitly in \cite{Ro04}, anonymously in \cite{RW} and pseudonymously in \cite{APT11}. In \cite[Definition 3.5]{Wi12}, Winter uses a slightly different version using the incomplete Cuntz semigroup $W(A)$ in place of $\Cu(A)$, and replacing (\ref{e5.1}) by $k[b]\leq [a]\leq (k+1)[b]$. This condition is formally stronger than almost-divisibility, though it is equivalent to it in the presence of strict comparison. Without strict comparison it is open whether these two notions are the same. Note too that in \cite[Remark 3.13, Definition 3.7]{RR13}, Robert and R\o{}rdam refer to a $\sigma$-unital $C^*$-algebra $A$ as almost divisible if the class given by a strictly positive element of $A$ is almost divisible.  } but based on developments in the setting of abstract Cuntz semigroups (see \cite{APT18}, for example) the prevailing definition today is that $A$ is \emph{almost divisible} if, for all $[a] \in \Cu(A)$, $k \in \N$ and $\epsilon > 0$, there exists $[b] \in \Cu(A)$ such that
\begin{equation}\label{e5.1}
	k[b] \leq [a] \quad \mbox{and} \quad [(a-\epsilon)_+] \leq (k+1)[b]. 
\end{equation}

In the presence of strict comparison, almost-divisibility, some of its earlier reformulations, and various other properties with similar flavours all become equivalent. We collect some of these facts below in the simple unital case, for completeness, and to set the scene for the discussion which follows. None of these are our results; most are a combination of observations known to experts which can be a little tricky to extract from the literature (a number of the implications, in the language and greater generality of Cu-semigroups, are contained in \cite[Proposition 2.11]{Th19}); the harder implication (iii)$\Rightarrow$(i) is due to Leonel Robert (communicated to us by Hannes Thiel).

\begin{proposition} \label{prop:Divisibility}
	Let $A$ be a simple, separable, unital, non-elementary, stably finite $C^*$-algebra with strict comparison and such that all $2$-quasitraces are traces.
The following are equivalent:
\begin{enumerate}[(i)]
	\item[(i)] $\Cu(A) \cong \Cu(A\otimes \mathcal Z)$.
	\item[(ii)] $A$ is almost divisible.
	\item[(ii$'$)] $A$ is almost divisible in the sense of \cite[Definition 3.5(a)]{Wi12}.
	\item[(iii)] For some $m \in \N_0$, $A$ is $m$-almost divisible, i.e., for all $[a] \in \Cu(A)$, $k \in \N$ and $\epsilon > 0$, there exists $[b] \in \Cu(A)$ such that
\begin{equation}
	k[b] \leq [a] \quad \mbox{and} \quad [(a-\epsilon)_+] \leq (k+1)(m+1)[b]. 
\end{equation}
\item[(iii$'$)] For some $m\in\N_0$, $A$ is $m$-almost divisible in the sense of \cite[Definition 3.5(a)]{Wi12}.
	\item[(iv)] $A$ is tracially almost divisible in the sense of \cite[Definition 3.5(ii)]{Wi12}, i.e., for every positive contraction $a\in (A\otimes M_n)_+$, $n,k\in\mathbb N$ and $\epsilon>0$, there exists a c.p.c.\ order zero map $\psi:M_k\rightarrow \overline{a(A\otimes M_n)a}$ such that 
\begin{equation}
\tau(\psi(1_{M_k}))\geq \tau(a)-\epsilon,\quad \tau\in T(A).
\end{equation} 
	\item[(v)] For some $m \in \N_0$, $A$ is tracially $m$-almost divisible in the sense of \cite[Definition 3.5(ii)]{Wi12}, i.e., for every positive contraction $a\in (A\otimes M_n)_+$, $n,k\in\mathbb N$ and $\epsilon>0$, there exists a c.p.c.\ order zero map $\psi:M_k\rightarrow \overline{a(A\otimes M_n)a}$ such that 
\begin{equation}
\tau(\psi(1_{M_k}))\geq\frac{1}{m+1}\tau(a)-\epsilon,\quad \tau\in T(A).
\end{equation}
\item[(vi)] For every function $f:F(\Cu(A)) \rightarrow [0,1]$ that is additive, order-preserving, homogeneous (with respect to positive
scalars), and lower semicontinuous, there exists $[b] \in \Cu(A)$ such that $f(\phi) = \phi([b])$ for all $\phi \in F(\Cu(A))$.   \label{item:AllRanks} 
\end{enumerate}
\end{proposition}

\begin{remark}
Condition (vi) above is one formulation of the `rank problem'. A positive operator $a\in (A\otimes\mathcal K)_+$, gives rise to a functional $\hat{a}:F(\Cu(A))\rightarrow[0,\infty]$ which measures the rank of $a$ with respect to each (lower semicontinuous quasi-)trace on $A$.  The rank problem asks whether all possible such functionals are realised by positive operators in $A\otimes\mathcal K$. It was first systematicaly investigated in \cite{DT}, and given heavy impetus by Nate Brown who promoted the viewpoint that the rank problem stands analogous to the fact that every possible trace value is obtained by a projection in a II$_1$ factor.  Dramatic progress has recently been obtained by Thiel in \cite{Th19}, who proved (independently of strict comparison) that (vi) holds for all simple, separable, unital, $C^*$-algebras with stable rank one.
\end{remark}

\begin{proof}[Proof of Proposition \ref{prop:Divisibility}.] 
(ii)$\Rightarrow$(iii), (ii$'$)$\Rightarrow$(iii$'$) and (iv)$\Rightarrow$(v) are all trivial consequences of the definitions. All hold without assuming strict comparison. (ii$'$)$\Rightarrow$(iv) is \cite[Proposition 3.8]{Wi12}, and also does not require strict comparison. (ii$'$)$\Rightarrow$(ii) and (iii$'$)$\Rightarrow$(iii) are almost immediate.\footnote{For (ii$'$)$\Rightarrow$(ii), given a positive element $a\in A\otimes \mathcal K$, $k\in\mathbb N$ and $\epsilon>0$, note that there exists $[b] \in W(A)$ between $[(a-\epsilon)_+]$ and $[a]$ by \cite[Lemma 3.2.7]{APT18}, so $[b]$ can be divided by (ii$'$), giving $[c]$ with $k[c]\leq [b] \leq[a]$ and $[(a-\epsilon)_+]\leq[b] \leq (k+1)[c]$ as required. The implication (iii$'$)$\Rightarrow$(iii) works in the same way.} (i)$\Rightarrow$(ii$'$) essentially goes back to \cite{Ro04}; it is recorded in \cite[Proposition 3.7]{Wi12}.

The implication (ii)$\Rightarrow$(i) is by now a folklore fact, and (the corresponding version using (ii$'$)) was certainly known to Winter at the time of writing of \cite{Wi12}.  One way to obtain it is as follows. If $A$ is almost divisible and has strict comparison then by \cite[Theorems 7.3.11 and 7.6.7]{APT18}, $\Cu(A) \cong (\Cu(A)_c\setminus\{0\}) \amalg L(F(\Cu(A)))$ (using the notation of \cite{APT18}, that is, $\Cu(A)_c$ is the set of compact elements of $\Cu(A)$, while $L(F(\Cu(A)))$ is the set of lower semicontinuous morphisms from $F(\Cu(A))$ to $[0,\infty]$ described in (vi)).
The set of $2$-quasitraces on $A$ identifies with $F(\Cu(A))$, and by \cite[Theorem 3.5]{BC09}, $\Cu(A)_c\cong V(A)$ --- the Murray von Neumann semigroup of $A$ --- so this description of $\Cu(A)$ agrees with the description of $\Cu(A\otimes \mathcal Z)$ given in \cite[Theorem 7.3.1]{APT18}.  This description also demonstrates (ii)$\Rightarrow$(vi).

(iii) $\Rightarrow$ (ii): Here is an argument (due to Leonel Robert and communicated to us by Hannes Thiel) to get almost-divisibility assuming $m$-almost-divisibility and strict comparison. Fix a contraction $a \in (A\otimes \mathcal K)_+$ to be `almost divided' by $k\in\mathbb N$, and let $\epsilon>0$. Since $C^*(a)\cong C(\sigma(a))$ has nuclear dimension at most $1$, and $A$ has $m$-almost-divisibility and strict comparison, there is a unital $^*$-homomorphism $\phi:\mathcal Z \to (A_\omega \cap C^*(a)')/\{a\}^\perp$ by \cite[Corollary 7.6]{RT17}.
Lifting $\phi$ to a sequence $(\phi_n)_{n=1}^\infty$ of $^*$-linear maps $\mathcal Z \to A$, and taking a positive contraction $h \in \mathcal Z$ such that $k[h] \leq [1_{\mathcal Z}] \leq (k+1)[h]$ (as given in \cite[Lemma 4.2]{Ro04}), then for $\delta>0$ sufficiently small and for $\omega$-almost all $n$ it will follow that $b_n\coloneqq (\phi_n(h)^{1/2}a\phi_n(h)^{1/2}-\delta)_+\in A_+$ satisfies $k[b_n] \leq [a]$ and $[(a-\epsilon)_+] \leq (k+1)[b_n]$.\footnote{To see this, set $b\coloneqq (b_n)_{n=1}^\infty \in A_\omega$.
Note that for $z \in \mathcal Z$ and $c\in C^*(a)$, $\phi(z)c$ is a well-defined element of $A_\omega$, and using this notation, $b=(\phi(h)^{1/2}a\phi(h)^{1/2}-\delta)_+ = (\phi(h)a-\delta)_+$. Throughout the rest of this footnote, given elements $x,y$ in a $C^*$-algebras, write $x\approx_\eta y$ to mean $\|x-y\|<\eta$.

For $[(a-\epsilon)_+]\leq(k+1)[b_n]$, it is standard that there exists $s \in M_{(k+1)\times 1}(\mathcal Z)$ such that $s^*(1_{k+1} \otimes h)s=1_{\mathcal Z}$ (this follows using compactness of $1_{\mathcal Z}$: for $1>\lambda>0$, the definition of Cuntz subequivalence gives $s_1\in M_{(k+1)\times 1}(\mathcal Z)$ such that $s_1^*(1_{k+1} \otimes h)s_1\approx_{\lambda/2} 1_{\mathcal Z}$, then using \cite[Lemma 2.2]{KR02}, for example, there exists $s_2\in \mathcal Z$ with $s_2^*s_1^*(1_{k+1} \otimes h)s_1s_2=(1_{\mathcal Z}-\lambda)_+=(1-\lambda)1_{\mathcal Z}$; now take $s=(1-\lambda)^{-1/2}s_1s_2$). Then, taking $\eta>0$ and $\delta>0$ sufficiently small that $2\eta+\delta \|s\|^2<\epsilon$, and a positive contraction $e \in C^*(a)$ such that $ea\approx_\eta a$, we have
\[ (\phi(s)e)^*(1_{k+1}\otimes \phi(h)a)(\phi(s)e) = \phi(s)^*(1_{k+1}\otimes \phi(h))\phi(s)eae = \phi(1_\mathcal Z)eae \approx_{2\eta} a. \]
It follows that $(\phi(s)e)^*(1_{k+1}\otimes b)(\phi(s)e) \approx_{\epsilon} a$. Thus if $(t_n)_{n=1}^\infty$ is a lift for $\phi(s)e$, then  for $\omega$-almost all $n$, $t_n^*(1_{k+1}\otimes b_n)t_n \approx_{\epsilon} a$, and so by \cite[Proposition 2.2]{Ro92}, $[(a-\epsilon)_+]\leq (k+1)[b_n]$. 

The argument is similar for $k[b_n]\leq[a]$. First fix $\eta_1>0$ and $s \in M_{1\times k}(\mathcal Z)$ such that $s^*s \approx_{\eta_1} 1_k\otimes h$. Thus, if $e\in C^*(a)$ is a positive contraction such that $ea\approx_{\eta_2} a$, where $\eta_2$ and $\eta_1$ are chosen small enough that $2\|s\|^2\eta_2+\eta_1<\delta$, where $\delta$ is as fixed in the previous paragraph, we have
\begin{align*}
(e\phi(s))^*a(e\phi(s)) &\approx_{2\|s\|^2\eta_2} \phi(s^*)a\phi(s) \\
&= \phi(s^*s)(1_k \otimes a) \\
&\approx_{\eta_1} 1_k\otimes \phi(h)a  = 1_k\otimes \phi(h)^{1/2}a\phi(h)^{1/2}.
\end{align*}
Thus if $(t_n)_{n=1}^\infty$ is a lift for $e\phi(s)$, then for $\omega$-almost all $n$, $t_n^*at_n \approx_\delta 1_k\otimes \phi_n(h)^{1/2}a\phi_n(h)^{1/2}$, and so by \cite[Proposition 2.2]{Ro92}, $k[b_n] \leq [t_n^*at_n] \leq [a]$.}

(v) $\Rightarrow$ (iii): Given $n,k\in\mathbb N$, a positive contraction $a\in A\otimes M_n$, and $0<\epsilon<1$ so that $(a-\epsilon)_+\neq 0$ (otherwise we can verify (iii) with $b=0$), let $f:[0,1]\rightarrow [0,1]$ be given by
$$
f(t)\coloneqq\begin{cases}t/\epsilon,&0\leq t\leq \epsilon,\\1,&t>\epsilon.\end{cases}
$$
Then $f(a)(a-\epsilon)_+=(a-\epsilon)_+$, so that $\tau(f(a))\geq d_\tau((a-\epsilon)_+)$ for all $\tau\in T(A)$.  Set $\delta:=\tfrac{1}{k}\min_{\tau\in T(A)}\tau(f(a))>0$, and use (v) to obtain a c.p.c.\ order zero map $\psi:M_k\rightarrow \overline{f(a)(A\otimes M_n)f(a)}$  with $\tau(\psi(1_k))\geq \tfrac{1}{m+1}\tau(f(a))-\delta$ for $\tau\in T(A)$. Then for $\tau\in T(A)$,
\begin{align}
(k+1)(m+1)d_\tau(\psi(e_{11}))&=\frac{(k+1)(m+1)}{k}d_\tau(\psi(1_k))\nonumber\\
&\geq\frac{k+1}{k}(\tau(f(a))-\delta)\nonumber\\
&\geq\tau(f(a))\geq d_\tau((a-\epsilon)_+).
\end{align}
Strict comparison gives $[(a-\epsilon)_+]\leq (k+1)(m+1)[\psi(e_{11})]$, while $k[\psi(e_{11})]=[\psi(1_k)]\leq [f(a)]\leq[a]$, by construction of $\psi$.

(vi) $\Rightarrow$ (ii):
Let $[a] \in \Cu(A)$, $k \in \N$ and $\epsilon > 0$. Consider the function $f:F(\Cu(A)) \rightarrow [0,1]$ defined by $f(\phi) = \tfrac{2}{2k+1}\phi([a])$ for all $\phi \in F(\Cu(A)) $. Then $f$ is additive, order-preserving, homogeneous and lower semicontinuous. Therefore, there exists $[b] \in \Cu(A)$ such that $f(\phi) = \phi([b])$ for all $\phi \in F(\Cu(A))$. 

By construction, we have
\begin{equation}
	\phi(k[b]) < \phi([a]) < \phi((k+1)[b]) < \infty
\end{equation}
for all $\phi \in F(\Cu(A))$ with $\phi([a]) < \infty$. By strict comparison, we have 
\begin{equation}
	k[b] \leq [a] \quad \mbox{and} \quad  [a] \leq (k+1)[b]. 
\end{equation} 
Hence, $A$ is almost divisible, noting that $[(a-\epsilon)_+] \leq [a]$. \end{proof}

In the terminology of \cite{Wi12}, the algebras covered by the previous proposition are called \emph{pure} $C^*$-algebras. 
In the direction of proving $\mathcal Z$-stability from such Cuntz semigroup conditions, a groundbreaking result was achieved by Winter in \cite{Wi12} for unital $C^*$-subalgebras of locally finite nuclear dimension;\footnote{$A$ has locally finite nuclear dimension if for all finite subsets $\mathcal F\subset A$ and $\epsilon>0$, there exists a $C^*$-algebra $B\subset A$ of finite nuclear dimension, which approximately contains $\mathcal F$ up to $\epsilon$.  This provides an abstract condition encompassing various classes of inductive limits, including for example all approximately subhomogeneous $C^*$-algebras.} we state the generalisation of this result to non-unital algebras below.

\begin{theorem}[{\cite[Corollary 7.2]{Wi12} and \cite[Corollary 8.8]{Ti14}}]\label{PureLFD}
Let $A$ be a separable, simple $C^*$-algebra with locally finite nuclear dimension.  Suppose $\Cu(A)\cong \Cu(A\otimes \Z)$.  Then $A\cong A\otimes \Z$.
\end{theorem}

The proof of Theorem \ref{PureLFD} goes in three main steps:
\begin{enumerate}[(1)]
\item Tracial $m$-divisibility, and a suitable dimensional weakening of comparison are obtained from finite nuclear dimension (\cite[Proposition 4.7]{Wi12} and \cite[Theorem 6.3]{Ti14}).\footnote{This is not applied to $A$, but to the subalgebras given by the hypothesis of locally finite nuclear dimension.}
\item In the presence of locally finite nuclear dimension, the  comparison property obtained in step (1), is shown to imply a weak form of comparison for the central sequence algebra (this is implicit in the proof of \cite[Propositions 6.5]{Wi12}, and stated explicitly in \cite[Proposition 5.4]{RT17}).
\item Tracial $m$-divisibility is shown to imply a divisibility condition for the central sequence algebra in the presence of locally finite nuclear dimension.
\end{enumerate}
The output of (2) and (3) then directly entails $\mathcal Z$-stability. With hindsight, we now recognise the output of step (3) as the uniform McDuff property.

\begin{theorem}[Winter]
Let $A$ be a simple separable $C^*$-algebra with $T(A)$ non-empty and compact and which is tracially $m$-almost divisible for some $m\in\mathbb N$ and has locally finite nuclear dimension.  Then $A$ has uniform property $\Gamma$.
\end{theorem}
\begin{proof}
This is \cite[Theorem 7.6]{Ti14} (\cite[Lemma 5.11]{Wi12} for the unital case), noting that the output of this theorem (and an $\epsilon$-test to move from $B$ with finite nuclear dimension to $A$ with locally finite nuclear dimension) produces the stronger McDuff type condition that for any $n\in\mathbb N$, there is an order zero map $\phi:M_n\rightarrow A_\omega\cap A'$ such that $\tau(\phi(1_{M_n}))=1$ for all $\tau\in T_\omega(A)$.\footnote{Actually \cite{Ti14} works with the sequence algebra $A_\infty$ relative commutant rather than the ultrapower, but this is equivalent (see Remark \ref{GammaMcDuffRem}).}
This conclusion is equivalent to condition (ii) of Theorem \ref{GammaMcDuff}, which immediately implies uniform property $\Gamma$.
\end{proof}

A major breakthrough was made by Matui and Sato in \cite{MS12}, which had the effect of removing the locally finite nuclear dimension hypothesis from (2) above.  Precisely they obtained a weak central version of comparison ---property (SI); see \cite[Definition 3.3]{Sa10} or \cite[Definition 2.6]{KR14} for the definition --- from comparison (and as Kirchberg and R\o{}rdam show, this even works for various weak versions of comparison as the hypothesis).  As a consequence, in the presence of an appropriate central divsibility condition, Matui and Sato deduce (iii)$\Rightarrow$(ii) in the Toms--Winter conjecture. The condition needed is precisely that the $C^*$-algebra is uniformly McDuff, which Matui and Sato showed was automatic in the presence of finitely many extremal traces.  Using Theorem \ref{GammaMcDuff}, the Toms--Winter conjecture holds in the presence of uniform property $\Gamma$.

\begin{theorem}
\label{thm:ZstableIffGammaComparison}
	Let $A$ be a simple, separable, unital, nuclear $C^*$-algebra with $T(A) \neq \emptyset$. The following are equivalent:
\begin{itemize}
	\item[(i)] $A \cong A \otimes \Z$.
	\item[(ii)] $A$ has strict comparison and uniform property $\Gamma$.
\end{itemize}
In particular the Toms--Winter conjecture holds under the hypothesis of uniform property $\Gamma$.
\end{theorem}
\begin{proof}
	(ii) $\Rightarrow$ (i): Suppose $A$ has uniform property $\Gamma$. Then $A$ has CPoU by \cite[Theorem 3.7]{CETWW}. Let $n \in \N$. By Theorem \ref{GammaMcDuff} and Remark \ref{GammaMcDuffRem}, there exists a c.p.c.\ order zero map $\phi:M_n
 \to A_\omega \cap A'$ such that $\tau(\phi(1_{M_n}))=1$ for all $\tau\in T_\omega(A)$.  These are the uniformly tracially large order zero maps in the sense of \cite[Definition 2.2]{TWW15}. When $A$ additionally has strict comparison of positive elements, it follows from   Matui and Sato's theorem in \cite{MS12} (this precise statement can be found as \cite[Theorem 2.6]{TWW15}) that $A$ is $\mathcal Z$-stable.
	
 (i) $\Rightarrow$ (ii): Suppose $A$ is $\Z$-stable. Then $A$ has strict comparison by \cite[Theorem 4.5]{Ro04} and has uniform property $\Gamma$ by \cite[Proposition 2.3]{CETWW}.
\end{proof}

At the moment, it remains an open problem whether all infinite dimensional, simple, nuclear $C^*$-algebras have uniform property $\Gamma$. Matui and Sato's work \cite{MS12} was subsequently extended from finitely many extremal traces, to compact extremal tracial boundaries of finite covering dimension \cite{Sa12,KR14,TWW15}. With hindsight, the reason is that these algebras always have uniform property $\Gamma$ -- a fact closely linked to Ozawa's investigation in \cite{Oz13} into the structure of the strict closures of such $C^*$-algebras.

\begin{proposition}[{cf.\ \cite[Theorem 4]{Oz13}}]
Let $A$ be a non-elementary simple, separable, unital, nuclear $C^*$-algebra with $\partial_eT(A)$ compact and finite-dimensional.  Then $A$ has uniform property $\Gamma$.
\end{proposition}
\begin{proof}
By \cite[Theorem 4.6]{TWW15}, $A$ admits uniformly tracially large c.p.c.\ order zero maps $M_n \to A_\infty \cap A'$; by Theorem \ref{GammaMcDuff} and Remark \ref{GammaMcDuffRem} (ii)$\Rightarrow$(iii), it follows that $A$ has uniform property $\Gamma$.
\end{proof}

Just as Matui and Sato were able to make a major breakthrough by removing the locally finite nuclear dimension hypothesis from Step (2) of the proof of Theorem \ref{PureLFD}, an important problem is whether it is possible to remove the locally finite nuclear dimension from Step 3 (and hence also from Theorem \ref{PureLFD}):

\begin{question}
Suppose that $A$ is a non-elementary simple, separable, unital and nuclear $C^*$-algebra with a tracial divisibility property. Does $A$ have uniform property $\Gamma$?
\end{question}

We round off this section by presenting a further family of examples of $C^*$-algebras with uniform property $\Gamma$. This family includes the non-$\Z$-stable Villadsen algebras of the first type constructed in \cite{Vi98}.\footnote{Note that Villadsen's ``second type'' examples from \cite{Vi99} also have uniform property $\Gamma$, since they are nuclear with unique trace.}

\begin{definition}
Let $X$ and $Y$ be compact Hausdorff spaces. A $^*$-homomorphism $\phi: C(X) \to M_n
 \otimes C(Y)$ is said to be \emph{diagonal} if there exist continuous functions $\lambda_1, \ldots, \lambda_n: Y \to X$ and matrix units $e_{rs} \in M_n$ 
 such that 
\begin{equation}
\phi(f) = \sum_{r=1}^n e_{rr} \otimes (f \circ \lambda_r), \quad f \in C(X).
\end{equation}
Matrix amplifications of diagonal maps are also said to be diagonal. 
\end{definition}

A $C^*$-algebra that can be represented as an inductive limit of (trivial) homogeneous $C^*$-algebra with diagonal connecting maps is called a \emph{diagonal AH} algebra (following \cite{EHT}, for example). For such algebras, uniform property $\Gamma$ can often be verified explicitly, as we now show. 

\begin{proposition}[{cf.\ \cite[Proposition 4.6.8]{Ev18}}]
\label{prop:DiagonalAHGamma}
Let $A$ be given by the inductive sequence
	\begin{equation}
	M_{n_1}
\otimes C(X_1) \stackrel{\phi_1}{\longrightarrow}  M_{n_2}
\otimes C(X_2) \stackrel{\phi_2}{\longrightarrow}  M_{n_3}
\otimes C(X_3) \ldots, 
	\end{equation}
where the connecting maps are diagonal and $n_i \rightarrow \infty$. Then $A$ has uniform property $\Gamma$ whenever $T(A)$ is a Bauer simplex.
\end{proposition}
\begin{proof}
	Since diagonal maps are unital, we have $n_i | n_{i+1}$. As $n_i \to \infty$, we can refine the inductive sequence and assume without loss of generality that $k_i \coloneqq  \frac{n_{i+1}}{n_i} \to \infty$.

Identifying $M_{n_{i+1}}
\otimes C(X_{i+1})$ with $M_{k_i}
\otimes C(X_{i+1}, M_{n_i})$, 
we can view $\phi_i $ as the map 
\begin{align}
	\phi_i : M_{n_i}
\otimes C(X_i) &\to M_{k_i}
\otimes C(X_{i+1}, M_{n_i}
) \notag \\
	f & \mapsto \sum_{r=1}^{k_i} e_{rr} \otimes  (f \circ \lambda_r) ,
	\end{align}
	for some continuous functions $\lambda_1, \ldots, \lambda_{k_i}: X_{i+1} \to X_i$ and matrix units $e_{rs} \in M_{k_i}$.

Writing $\mu_i: M_{n_i}
 \otimes C(X_i) \to A$ for the canonical maps into the inductive limit, we define 
	\begin{equation} 
	q_i \coloneqq  \sum_{r=1}^{\lfloor \frac{k_i}{2} \rfloor} e_{rr} \otimes 1_{M_{n_i}}  \in M_{k_i} 
\otimes C(X_{i+1}, M_{n_i}) 
	\end{equation}
	and
	\begin{equation}
	 p_i \coloneqq  \mu_{i+1}(q_i) \in A.
	\end{equation}

By construction, $p_i$ is a projection in $A$ that commutes with the image of $\mu_{i}$. As the sequence $(p_i)_{i=1}^\infty$ is uniformly bounded, it follows that 
	$\lim_{i \to \infty} \| [p_i, a] \|= 0$ for all $a \in A$. 
		
Also by construction, we have that $\tau(q_i)\in [\frac{1}{2} - \frac{1}{k_i}, \frac{1}{2}]$ for all $\tau \in T(M_{n_{i+1}}
 \otimes C(X_{i+1}))$. Since every trace on $A$ pulls back to a trace on $M_{n_{i}}
 \otimes C(X_i)$ and $k_i \to \infty$, it follows 
	\begin{equation}
	\lim_{i \to \infty}\sup_{\tau \in T(A)} | \tau(p_i) - \tfrac{1}{2}| = 0.
	\end{equation}
	Let $p$ be the element of $A^\omega$ induced by the sequence $(p_i)_{i=1}^\infty$. Then $p \in A^\omega \cap A'$ and $\tau(p) = \frac{1}{2}$ for all $\tau \in T_\omega (A)$. Since we are assuming that $T(A)$ is a Bauer simplex, we have
	\begin{equation}
	\tau(pa) = \tau(p)\tau(a) = \frac{1}{2} \tau(a), \qquad a \in A, \tau \in T_\omega (A),
	\end{equation}
by Proposition \ref{prop:BauerFactorization}. It now follows from Proposition \ref{prop:GammaEquivalence}(ii) that $A$ has uniform property $\Gamma$.
\end{proof} 

	In particular, Villadsen's example in \cite{Vi98}, and more generally all of the ``Villadsen algebras of the first type'' constructed in \cite[Section 8]{TW09} (including non-$\mathcal Z$-stable examples) have uniform property $\Gamma$ since these are diagonal AH algebras with Bauer trace simplices (see \cite[Section 8]{TW09} and the computations of \cite[Theorem 4.1]{To08b}).

\section{Classification by embeddings, revisited}
\label{sec:CBErevisited}

 In 2015, the classification of simple, separable, unital $C^*$-algebras of finite nuclear dimension in the UCT class was completed by combining \cite{GLN,EGLN15,TWW17} (themselves building on an enormous body of earlier work) through a combination of classification theorems for $C^*$-algebras with good tracial approximations, and abstract machinary for accessing these.  Now, using \cite{CETWW}, $\mathcal Z$-stability can be used in place of finite nuclear dimension to access classification.  We end the paper by explaining how $\Z$-stability and finite nuclear dimension are used in the tracial approximation approach to classification, and how the CPoU methods of \cite{CETWW} fit in.\footnote{CPoU (and hence uniform property $\Gamma$ as the tool for accessing it) is also crucial to the new abstract approaches to classification \cite{CGSTW,CGSTW2} emerging from Schafhauser's approach to the quasidiagonality theorem \cite{S}, and AF-embeddings \cite{S2}.} Recall that for simple $C^*$-algebras, both of the equivalent conditions of finite nuclear dimension or $\Z$-stability give rise to a dichotomy: such algebras are either stably finite or purely infinite.\footnote{Both dichotomy results rely on work of Kirchberg. The finite nuclear dimension result was obtained by Winter and Zacharias as \cite[Theorem 5.4]{WZ10}; the $\Z$-stability result is recorded in \cite[Theorem 4.1.10]{Ro02}.} The purely infinite classification theorem was established by Kirchberg and by Phillips (\cite{Kir:Unpublished,Phi00}) in the 90s, so in the rest of this section we restrict to stably finite $C^*$-algebras.

Internal approximations by building blocks are familiar in operator algebras, originating in Murray and von Neumann's work on hyperfinite II$_1$ factors and appearing prominently in the subsequent classification of approximately finite dimensional $C^*$-algebras \cite{Br72,El76}, and various more general inductive limit models.
In \cite{Li01}, (building on Popa's earlier local quantisation in the setting of $C^*$-algebras \cite{Po97}), Lin developed a new kind of internal approximation property for $C^*$-algebras, now known as tracial approximation.  Let us recall the precise definitions in a form suitable for our discussion. 

\begin{definition}[{cf.\ \cite{Li01,EN08}}]\label{TABlah}
Let $\mathcal{S}$ denote a class of separable unital $C^*$-algebras closed under isomorphism. A simple unital $C^*$-algebra $A$ is \textit{tracially approximately} $\mathcal{S}$ (TA$\mathcal{S}$) if for every finite subset $\mathcal{F} \subset A$, every $\epsilon > 0$, and every positive element $c\in A_+$, there is a projection $p \in A$ and a unital $C^*$-subalgebra $B \subset pAp$ with $1_B =p$ and $B \in \mathcal{S}$ such that
\begin{enumerate}[(i)]
	\item $\| pa -ap \| < \epsilon$ for all $a \in \mathcal{F}$,
	\item $\mathrm{dist}(pap,B) < \epsilon$ for all $a \in \mathcal{F}$, and
	\item $1_A - p$ is Murray--von Neumann equivalent to a projection in $\overline{cAc}$.
\end{enumerate}
When $T(A)\neq\emptyset$ and $A$ has strict comparison of positive elements by traces (e.g., if it is unital, exact and $\mathcal Z$-stable), one can replace (iii) by the following condition (where the tracial nature of the internal approximation is more evident).
\begin{enumerate}
\item[(iii$'$)] $\tau(p)>1-\epsilon$ for all traces $\tau \in T(A)$.
\end{enumerate}
A simple unital $C^*$-algebra $A$ is said to be \emph{rationally} TA$\mathcal S$, if $A\otimes\mathcal U$ is TA$\mathcal S$ for every UHF algebra $\mathcal U$ of infinite type.
\end{definition} 

Lin's original work \cite{Li01} focused on tracially AF\footnote{Note that, within the class of unital simple $C^*$-algebras, the property of being TAF is equivalent to Lin's later notion of tracial topological rank zero of \cite[Definition 3.1]{L:PLMS01}; see \cite[Theorem 7.1]{L:PLMS01}.} (TAF) $C^*$-algebras, firstly giving an abstract characterisation of the universal UHF-algebra within the class of simple, separable, unital, nuclear, TAF $C^*$-algebras with the UCT, and then subsequently obtaining a complete classification of this class \cite{Lin04}.  Under trace space conditions, this theorem can be accessed abstractly through finite decomposition rank (a precursor of finite nuclear dimension from \cite{KW04}, which entails quasidiagonality of all traces); see \cite{Wi05}. 

Over time the class of algebras classified via by tracial approximation methods expanded, culminating in Gong--Lin--Niu's classification in \cite{GLN} of simple, separable, unital, $\mathcal Z$-stable, nuclear $C^*$-algebras with the UCT which are rationally TA$\mathcal S$, where $\mathcal S$ denotes the class of so called Elliott--Thomsen building blocks.\footnote{These are also known as point-line algebras or 1-dimensional NCCW complexes. The exact form of these building blocks is not needed here.}
(This class exhausts the Elliott invariant by \cite[Theorem 13.46]{GLN}.)
Here, $\mathcal Z$-stability plays a key role, as it allows asymptotic classification results to feed into Winter's strategy from \cite{Wi14} to pass from the TA$\mathcal S$ class to the (larger) $\mathcal Z$-stable, rationally TA$\mathcal S$ class.
On the other hand, finite nuclear dimension does not play a role in this argument, only $\Z$-stability; thus if one works with finite nuclear dimension as the regularity hypothesis in classification via tracial approximations, it is necessary to use Winter's long and difficult `$\Z$-stability' theorem from \cite{Wi12}.

The other major component in the tracial approximation approach to $C^*$-algebra classification is an abstract criterion for identifying the rationally TA$\mathcal S$ class. 
A new approach to problems of this nature was pioneered by Matui and Sato in \cite[Theorem 6.1]{MS14}, where they obtained rationally TAF approximations for a simple, nuclear, quasidiagonal $C^*$-algebra with a unique tracial state, directly from quasidiagonality.  Around the same time, Winter developed his `classification by embeddings' technique (\cite[Theorem 2.2]{Wi16}) which allows one to use finite nuclear dimension to convert certain tracial approximate factorisations to tracial approximations in the sense of Definition \ref{TABlah}.\footnote{See the statement of Theorem \ref{ClassEmbeddings} below for a more precise statement}  In \cite{EGLN15}, Elliott, Gong, Lin, and Niu used a quasidiagonality hypothesis in the spirit of Matui--Sato (the precise condition is quasidiagonality of all traces), to obtain the required input to the classification by embeddings theorem.
Simultaneously, this quasidiagonality hypothesis was shown to be automatic in the presence of the UCT \cite{TWW17}.  In summary, these papers come together to show that simple, separable, unital $C^*$-algebras of finite nuclear dimension in the UCT class are rationally TA$\mathcal S$, and hence classified by \cite{GLN}.  In this latter (classification by embeddings) part of the argument, finite nuclear dimension rather than $\mathcal Z$-stability is the regularity hypothesis being directly employed.

We now know, by the main results of \cite{CETWW,CE}, that finite nuclear dimension and $\Z$-stability are equivalent for simple, separable,  non-elementary, nuclear $C^*$-algebras. So following the outline above, to classify from the hypothesis of $\Z$-stability, we should obtain finite nuclear dimension from \cite{CETWW}, and feed this into the classification by embeddings theorem.  However it turns out, as we now show, that the methods of \cite{CETWW} and \cite{BBSTWW} also give a direct proof of the classification by embeddings theorem from $\mathcal Z$-stability without requiring the hypothesis of finite nuclear dimension, and so one does not need the full strength of finite nuclear dimension implies $\Z$-stability for such classification. This is inspired by Matui and Sato's strategy from \cite{MS14}.  Specifically, this has the effect of removing the finite nuclear dimension hypothesis from \cite[Theorem 2.2]{Wi16}.

\begin{theorem}\label{ClassEmbeddings}
	Let $\S$ be a class of separable, unital $C^*$-algebras which is closed under isomorphism and matrix amplifications.
	Let $A$ be a simple, separable, nuclear, unital $C^*$-algebra with $T(A) \neq \emptyset$ and let 
	\begin{equation}
	(A \overset{\sigma_i}{\longrightarrow} B_i \overset{\rho_i}{\longrightarrow} A)_{i=1}^\infty \label{eqn:System}
	\end{equation}
	be a system of maps with the following properties:
	\begin{itemize}
		\item[(i)] $B_i \in \S$, $i \in \mathbb{N}$,
		\item[(ii)] $\rho_i$ is an injective $^*$-homomorphism for each $i \in \mathbb{N}$,
		\item[(iii)] $\sigma_i$ is c.p.c.\ for each $i \in \mathbb{N}$,
		\item[(iv)] the map $\overline{\sigma}:A \rightarrow \prod_{\omega} B_i$ induced by the $\sigma_i$ is a unital $^*$-homomorphism,
		\item[(v)] $\sup \lbrace |\tau(\rho_i\circ\sigma_i(a)-a)| : \tau \in T(A)\rbrace \overset{i\rightarrow\omega}{\longrightarrow} 0$ for each $a \in A$.
	\end{itemize}
	Then, $A$ is rationally TA$\S$. 
\end{theorem}
\begin{proof}
Fix a UHF algebra $\mathcal U$ of infinite type. Without loss of generality, we may assume that (iii) is strengthened to $\sigma_i$ is u.c.p.\ for each $i\in\mathbb N$. Indeed, by assumption (iv) we have $\lim_{i \rightarrow \omega} \|\sigma_i(1_A) - 1_{B_i}\| = 0$. Therefore, for $\omega$-many $i$ (i.e., for all $i$ in some set $I \in \omega$), $\sigma_i(1_A)$ is invertible and $\lim_{i \rightarrow \omega} \|\sigma_i(1_A)^{-1} - 1_{B_i}\| = 0$. We can then replace $\sigma_i$ with the u.c.p.\ map $\sigma_i(1_A)^{-1/2}\sigma_i(\cdot)\sigma_i(1_A)^{-1/2}$ for $\omega$-many $i$ (and replace the remaining $\sigma_i$ by arbitrary u.c.p.\ maps), which does not change $\overline{\sigma}$, and thus does not change (iv) or (v). (Conditions (i)-(ii) are completely unaffected.)
	
In the rest of the proof, we shall view $A$ and $A_\omega \otimes \mathcal U$ as subalgebras of $A_\omega$ and $(A \otimes \mathcal U)_\omega$ respectively, in the natural ways.  Set $\phi_i \coloneqq  \rho_i \circ \sigma_i$ for $i \in \mathbb{N}$ and let $\Phi:A \rightarrow A_\omega$ be the map induced by the $\phi_i$.  Our assumptions ensure that $\Phi$ is a $^*$-homomorphism satisfying 
\begin{align}
\tau \circ \Phi &= \tau|_A, \qquad \tau \in T_\omega(A) \label{eqn:PhiPreservesTrace}.
\end{align}

 Let $\pi:A \rightarrow M_2((A \otimes \mathcal U)_\omega)$ be the $^*$-homomorphism defined by 
\begin{equation}\label{e6.3}
\pi(a) \coloneqq  \begin{pmatrix}
\Phi(a) \otimes 1_{\mathcal U} & 0 \\
0 & a \otimes 1_{\mathcal U}
\end{pmatrix},\quad a\in A.
\end{equation}
Set 
\begin{equation}\label{e6.4}
C \coloneqq  M_2((A \otimes {\mathcal U})_\omega) \cap \pi(A)' \cap \lbrace 1_{M_2((A \otimes {\mathcal U})_\omega)} - \pi(1_A)\rbrace^\perp.
\end{equation} 
Applying \cite[Lemma 4.7]{CETWW},\footnote{Checking the conditions carefully: $A$ is separable, unital, and nuclear; $B_n \coloneqq  A \otimes \mathcal U$ is simple, separable, unital, and $\mathcal{Z}$-stable; $B_n$ has $T(B_n) = QT(B_n)$ and CPoU by virtue of being additionally nuclear \cite[Theorem 3.7]{CETWW}, and $T(B_n) = T(A) \neq \emptyset$ is assumed (the implicit assumption that $T(B_n)$ is compact, contained in the definition of CPoU follows from unitality); $\pi$ is a $^*$-homomorphism. For any non-zero $a\in A$, since $M_2 \otimes A \otimes \mathcal U$ is simple and $\pi(a) \geq \left(\begin{array}{cc} 0&0\\0&a\otimes 1_{\mathcal U}\end{array}\right)$, it follows that $\pi(a)$ is full.} we see that $C$ has strict comparison of positive elements by bounded traces and $T(C)$ is the closed convex hull of the set of traces of the form $\tau(\pi(a)\cdot)$ where $\tau \in T(M_2((A \otimes \mathcal U)_\omega))$ and $a\in A_+$ satisfies $\tau(\pi(a)) = 1$. 
	
Since $A \otimes \mathcal U$ is separable and $\mathcal{Z}$-stable, $T_\omega(M_2((A \otimes \mathcal U)))$ is weak$^*$-dense in $T(M_2((A \otimes \mathcal U)_\omega))$ by \cite[Theorem 8]{Oz13}. Therefore, $T(C)$ is the closed convex hull of the set of traces of the form $\tau(\pi(a)\cdot)$ where $\tau \in T_\omega(M_2((A \otimes \mathcal U)))$ and $a\in A_+$ satisfies $\tau(\pi(a)) = 1$.
	
Fix $0 < \epsilon < 1$. Let $p$ be a projection in $\mathcal U$ with $1 - \epsilon < \tr_\mathcal U(p) < 1$. Define projections $h_1, h_2 \in C$ by 
\begin{align}
h_1 \coloneqq  \begin{pmatrix}
\Phi(1_A) \otimes p & 0 \\
0 & 0
\end{pmatrix},
&&h_2 \coloneqq  \begin{pmatrix}
0 & 0 \\
0 & 1_A \otimes 1_\mathcal U
\end{pmatrix}.
\end{align}

We wish to argue now that $\tau(h_1) < \tau(h_2)$ for all $\tau \in T(C)$.
First, consider a trace of the form $\mu=\tau(\pi(a)\cdot )$ where $\tau \in T_\omega(M_2((A\otimes \mathcal U)))$ and $a \in A_+$ satisfies $\tau(\pi(a))=1$. 
As $M_2$ and $\mathcal U$ have unique tracial states $\tr_{M_2}$ and $\tr_\mathcal U$ respectively, 
\begin{equation} \tau = \tr_{M_2} \otimes \widetilde \tau \otimes \tr_\mathcal U\end{equation}
for some $\widetilde{\tau} \in T_\omega(A)$. 
	
For $b \in \mathcal U$, we compute that
\begin{eqnarray}
\notag
\tau \left( \begin{pmatrix}
\Phi(a) \otimes b & 0 \\
0 & 0
\end{pmatrix}\right) 
&=& \frac{1}{2}\widetilde{\tau}(\Phi(a))\tr_\mathcal U(b)\\ 
&\stackrel{(\ref{eqn:PhiPreservesTrace})}{=}& \frac{1}{2}\widetilde{\tau}(a)\tr_\mathcal U(b),
\end{eqnarray}
and that   
\begin{align}
\tau \left( \begin{pmatrix}
0 & 0 \\
0 & a \otimes b
\end{pmatrix}\right) &= \frac{1}{2}\widetilde{\tau}(a)\tr_\mathcal U(b).
\end{align}
It follows that $\widetilde{\tau}(a) = \widetilde{\tau}(a)\tr_\mathcal U(1_Q)= \tau(\pi(a)) = 1$. Thus $\tau(\pi(a)h_1) = \tfrac{1}{2}\tr_\mathcal U(p)$, and $\tau(\pi(a)h_2) = \tfrac{1}{2}$.
Hence, we have 
\begin{equation}
\mu(h_2-h_1) = \tau(\pi(a)h_2) - \tau(\pi(a)h_1) = \tfrac{1}{2}(1 - \tr_\mathcal U(p)).
\end{equation}
By convexity and continuity, it now follows that $\mu(h_1)-\mu(h_2)=\tfrac12(1-\tr_\mathcal U(p))>0$ for all $\mu \in T(C)$.
	
Since $h_1$ and $h_2$ are projections, we get
\begin{equation} d_\mu(h_2) = \mu(h_2) > \mu(h_1) = d_\mu(h_1) \end{equation}
for all $\mu \in T(C)$. By strict comparison, $h_1$ is Cuntz subequivalent to $h_2$. For projections, Cuntz subequivalence implies Murray--von Neumann subequivalence (by \cite[Proposition 2.1]{Ro92}). 
Therefore, there exists $v \in C$ with $v^*v = h_1$ and $vv^* \leq h_2$. 
These conditions force $v$ to take the form
\begin{equation}
v = \begin{pmatrix}
0 & 0 \\
w & 0 
\end{pmatrix},
\end{equation}
where $w \in (A \otimes \mathcal U)_\omega$ satisfies $w^*w = \Phi(1_A) \otimes p$. The condition that $v$ commutes with $\pi(A)$ implies that
\begin{equation}
w(\Phi(a) \otimes 1_\mathcal U) = (a \otimes 1_{\mathcal U})w, \qquad a \in A, \label{eqn:wIntertwines}
\end{equation}
from which we deduce that $ww^*$ commutes with $A \otimes 1_\mathcal U$.
	
The element $\Phi(1_A) \otimes p$ of $(A \otimes \mathcal U)_\omega$ is represented by the sequence of positive elements $(\phi_i(1_A) \otimes p)_{i =1}^\infty$. In fact, 
\begin{equation} \label{eqn:Embeddings1}
\phi_i(1_A)\otimes p = \rho_i(1_{B_i}) \otimes p
\end{equation} since $\sigma_i$ is unital, and this is a projection since $\rho_i$ is a $^*$-homomorphism.
By standard stability results (e.g., combining the proofs of \cite[Propositon 2.2.6 and Lemma 6.3.1]{Ro00}), the partial isometry $w \in (A \otimes \mathcal U)_\omega$ with source projection $\Phi(1_A) \otimes p$ lifts to a sequence $(w_i)_{i=1}^\infty$ of partial isometries satisfying $w_i^*w_i = \phi_i(1_A) \otimes p=\rho_i(1_{B_i})\otimes p$ for each $i \in \mathbb{N}$.\
	
For $i \in \mathbb{N}$, define $\hat{\rho}_i: B_i \to A \otimes \mathcal U$ by 
\begin{equation}
\hat{\rho}_i (b)\coloneqq  w_i(\rho_i(b) \otimes 1_{\mathcal U})w_i^*, \qquad b \in B_i,
\end{equation}
which is a $^*$-homomorphism since $\rho_i$ is a $^*$-homomorphism and $w_i^*w_i$ commutes with $\rho_i(B_i)\otimes 1_{\mathcal U}$.
Also, since $\rho_i$ is injective and $w_i^*\hat\rho_i(b)w_i = \rho_i(b) \otimes p$, it follows that $\hat\rho_i$ is injective.

Set $D_i \coloneqq  \hat{\rho}_i (B_i) = w_i(\rho_i(B_i) \otimes 1_\mathcal U)w_i^*$. This is a $C^*$-subalgebra of $A \otimes \mathcal U$ isomorphic to $B_i$, so that $D_i \in \S$.	
For $i \in \mathbb{N}$, set $q_i \coloneqq  w_iw_i^* = 1_{D_i}$. Then, $q_i$ is a projection in $A \otimes \mathcal U$ and, for any $\tau \in T(A \otimes \mathcal U)$, 
\begin{align}
\notag
\tau(q_i) &= \tau(w_iw_i^*) \\
\notag
&= \tau(w_i^*w_i)\\
\notag
&= \tau(\phi_i(1_A) \otimes p)\\
\notag
&= \tau(\phi_i(1_A) \otimes 1_\mathcal U)\tr_\mathcal U(p)\\
&\xrightarrow{i \to \omega} \tr_{\mathcal U}(p),
\end{align}
using (v) in the last line, and thereby obtaining uniform convergence over $\tau \in T(A \otimes \mathcal U)$. In addition, since $ww^*$ commutes with $a \otimes 1_{\mathcal U}$ for all $a \in A$, we have 
\begin{equation}
\|q_i(a \otimes 1_\mathcal U) - (a \otimes 1_\mathcal U)q_i\| \xrightarrow{i \to \omega} 0, \qquad a \in A.
\end{equation}
Finally for $a \in A$, working in $(A\otimes\mathcal U)_\omega$, we have
\begin{eqnarray}
\notag
(q_i(a\otimes 1_{\mathcal U}) q_i)_{i=1}^\infty & = & ww^*(a \otimes 1_\mathcal U)ww^* \\
\notag
& \stackrel{(\ref{eqn:wIntertwines})}{=} & ww^*w(\Phi(a) \otimes 1_\mathcal U)w^* \\
\notag
& = &  w(\Phi(a) \otimes 1_\mathcal U)w^* \\
&\in& \prod_\omega D_i.
\end{eqnarray}
Thus, 
\begin{equation}
\mathrm{dist}(q_i(a \otimes 1_\mathcal U)q_i, D_i) \xrightarrow{i \to \omega}
0
\end{equation}
for all $a \in A$.
	
Therefore, for any finite set $\mathcal{F} \subset A$ and for $\omega$-many $i$, we have
\begin{enumerate}
	\item $\|q_i(a \otimes 1_\mathcal U) - (a \otimes 1_\mathcal U)q_i\| < \epsilon$ for all $a \in \mathcal{F}$,
	\item $\mathrm{dist}(q_i(a \otimes 1_\Q)q_i, D_i) < \epsilon$ for all $a \in \mathcal{F}$, and
	\item $\tau(q_i) > 1 - \epsilon$ for all $\tau \in T(A \otimes \mathcal U)$. 
\end{enumerate} 
This almost says that $A \otimes \mathcal U$ satisfies the definition of TA$\S$, except that the finite set is in $A\otimes 1_{\mathcal U}$ rather than in $A \otimes \mathcal U$.  

Now, given a finite subset $\mathcal F\subset A\otimes\mathcal U$ and $\delta>0$, we can find a factorisation $\mathcal U= F\otimes\mathcal V$, where $F$ is a full matrix algebra and $\mathcal V\cong\mathcal U$, together with a finite subset $\mathcal F'$ of $A \otimes F \otimes 1_{\mathcal V}$ such that every element in $\mathcal F$ is within $\delta$ of a corresponding element of $\mathcal F'$.  
Hence (by using $A \otimes F$ in place of $A$, and since $\mathcal S$ is closed under matrix amplifications), it follows that $A \otimes \mathcal U$ is TA$\mathcal S$.
\end{proof}

\begin{remark}
While the argument above and the proof of finite nuclear dimension from $\Z$-stability in \cite{CETWW} both rely on the detailed analysis of relative commutants in ultrapowers \`a la Matui--Sato from \cite{MS12,MS14}, in fact the argument above is considerably more straightforward for two reasons.  Firstly, since the map $\pi$ in \eqref{e6.3} is a $^*$-homomorphism rather than an order zero map, we only need the versions of property (SI) for $^*$-homomorphisms (\cite[Lemma 3.2]{MS14}) rather than the more elaborate order zero map property (SI) techniques of \cite[Section 4]{BBSTWW}.  Secondly, obtaining finite nuclear dimension from $\Z$-stability requires a relatively sophisticated existence theorem (\cite[Lemma 5.2]{CETWW}, \cite[Lemma 7.4]{BBSTWW}) to give rise to the eventual nuclear dimension approximations; in contrast, the corresponding part of Theorem \ref{ClassEmbeddings} is the existence of $\Phi$, which is automatic from the assumptions.
\end{remark}

\begin{remark}\label{rem:Connes}
There is another point which is worthy of note in contrasting the proof of Theorem \ref{ClassEmbeddings} with $\Z$-stability implies finite nuclear dimension.  Both come down to comparison results for relative commutants of the form $C$ in (\ref{e6.4}).  In the nuclear dimension argument of \cite{CETWW} and \cite{BBSTWW}, one takes two positive elements $h_1,h_2$ in such a relative commutant of full spectrum, such that $\tau(h_1^k)=\tau(h_2^k)$ for all $k\in\mathbb N$ and all $\tau\in T(C)$, and obtains unitary equivalence of $h_1$ and $h_2$ (see \cite[Theorem 5.1]{BBSTWW}).  By contrast, in the argument above, one has projections $h_1,h_2\in C$ with $\tau(h_1)<\tau(h_2)$ for all $\tau\in T(C)$, and obtains Murray--von Neumann subequivalence $h_1\precsim h_2$ from strict comparison of $C$.  If it were possible to deduce that projections $h_1,h_2\in C$ with $\tau(h_1)\leq \tau(h_2)$ satisfied $h_1\precsim h_2$, then the argument of Theorem \ref{ClassEmbeddings} would give that $A\otimes\mathcal Q$ is A$\mathcal S$, not just TA$\mathcal S$.  In a $C^*$-algebraic setting such a comparison result for projections (as opposed to positive elements of full spectrum) is too much to hope for; indeed TAF $C^*$-algebras need not be AF.\footnote{\cite[Theorem 2.1]{Lin03} provides many examples, such as the irrational rotation algebra, which is not AF since its $K_1$ group is nonzero.} However, in finite von Neumann algebras, tracial data does completely determine the Murray--von Neumann order on projections, and as such one can perform the final steps of Connes' original approach to his hyperfiniteness theorem in a very similar fashion to the proof of Theorem \ref{ClassEmbeddings}. 

Precisely, let $\mathcal M$ be an injective II$_1$ factor with separable predual. In earlier parts of his argument, Connes shows that $\M$ is McDuff, i.e., $\mathcal M\cong \mathcal M\,\overline{\otimes}\,\mathcal R$, where $\mathcal R$ is the hyperfinite II$_1$ factor \cite[Theorem 5.1]{Co76}, and there is an embedding $\Phi:\M\hookrightarrow \mathcal R^\omega$ \cite[Lemma 5.22]{Co76}.  To obtain hyperfiniteness one needs to show that the first factor embedding $\iota:\mathcal M\hookrightarrow (\mathcal M\,\overline{\otimes}\,\mathcal R)^\omega$ is unitarily equivalent to $\Phi$. This can be deduced from the approximately inner flip on $\mathcal M$ (obtained by Connes as \cite[Lemma 5.25]{Co76} and the direction 3$\Rightarrow$2 of \cite[Theorem 5.1]{Co76}), and it shows that with $\pi:\mathcal M\rightarrow M_2(\mathcal M\,\overline{\otimes}\,\mathcal R)^\omega$ given as the direct sum of $\iota$ and $\Phi$, the projections $\begin{pmatrix}1&0\\0&0\end{pmatrix}$ and $\begin{pmatrix}0&0\\0&1\end{pmatrix}$ are Murray--von Neumann equivalent in this relative commutant.\footnote{After the fact (i.e. once we have hyperfiniteness), it turns out that this relative commutant is a factor; this is deduced as the relative commutant of a unital matrix subalgebra of a factor is a factor (see for example the proof of \cite[Theorem 4.18]{Tak3}).} We note, finally, that by unpicking the analysis of traces on the algebra $C$ in \eqref{e6.4} (performed behind the scenes in \cite[Lemma 4.7]{CETWW}), one finds that this relies heavily on the analysis of traces on von Neumann relative commutants of injective von Neumann algebras. 
\end{remark}

\newcommand{\cstar}{$C^*$}

\end{document}